\documentclass[11pt]{article}
\usepackage{amsmath,amsfonts,amssymb}
\usepackage{bbm}
\usepackage{cases}
\usepackage[hidelinks]{hyperref}
\usepackage{graphicx}
\usepackage{setspace}
\usepackage{xcolor}

\def\<{{\langle}}
\def\>{{\rangle}}
\allowdisplaybreaks[4]

\newtheorem{theorem}{Theorem}[section]
\newtheorem{lemma}[theorem]{Lemma}
\newtheorem{corollary}[theorem]{Corollary}
\newtheorem{proposition}[theorem]{Proposition}

\newtheorem{remark}[theorem]{Remark}

\newtheorem{definition}[theorem]{Definition}

\renewcommand{\theequation}{\arabic{section}.\arabic{equation}}
\newenvironment{proof}{{ \noindent {\bf Proof}.}}{\hfill$\blacksquare$\par}

\begin{document}
	
	\title{\vspace*{-1.5cm}
		Stochastic Stefan problem on moving hypersurfaces: an approach  by a new framework of nonhomogeneous monotonicity}
	
	\author{Tianyi Pan$^{1}$, Wei Wang$^{1}$, Jianliang Zhai$^{1}$, Tusheng Zhang$^{1,2}$}
	\footnotetext[1]{\, School of Mathematics, University of Science and Technology of China, Hefei, China. Email:pty0512@mail.ustc.edu.cn(Tianyi Pan), ww12358@mail.ustc.edu.cn(Wei Wang), zhaijl@ustc.edu.cn(Jianliang Zhai).}
	\footnotetext[2]{\, Department of Mathematics, University of Manchester, Oxford Road, Manchester
		M13 9PL, England, U.K. Email: tusheng.zhang@manchester.ac.uk.}
	\date{}
	\maketitle
	
	\begin{abstract}
		In this paper, we consider stochastic Stefan problem on moving hypersurfaces, which is written as
		\begin{numcases}{}\label{sfe}
			\partial^\bullet X_t+X_t\nabla_{\Gamma_t}\cdot v(t)=\Delta_{\Gamma_t}\beta(X_t)+B(t,X_t)\partial^\bullet W_t,\ \text{on}\ Q_T
			\\
			X_0=x_0\ \text{on}\ \Gamma_0,\nonumber
		\end{numcases}
		where $Q_T$ is a $n+1$-dimensional space-time hypersurface in $\mathbb{R}^+\times\mathbb{R}^{n+1}$ of the form
		\begin{align}\label{09022055}
			Q_T=\mathop{\scalebox{1.5}[1.5]{$\cup$}}_{t\in[0,T]} \{t\}\times\Gamma_t,
		\end{align}
		where for every $t\in[0,T]$, the section $\Gamma_t$ is a $n$-dimensional hypersurface in $\mathbb{R}^{n+1}$ and is diffeomorphic to $\Gamma_0$. The symbol $v$ denotes the speed of the movement of the hypersurface. $\partial^\bullet$ is the material derivative on $Q_T$, and the function $\beta$ is given by
\begin{numcases}{\beta(r)=}\label{1}
		\nonumber ar,\ \ r<0,\\
		0,\ \ 0\leq r\leq \rho,\\
		\nonumber b(r-\rho),\ \ r>\rho,
	\end{numcases}
	for some $a,b,\rho>0$.
		\par
		The purpose of this paper is to establish the well-posedness of stochastic partial differential equation (\ref{sfe}). Through a specially designed transformation,  it turns out we need to solve stochastic partial differential equations on a fixed domain with a new kind of nonhomogeneous monotonicity involving a family of time-dependent operators. This new class of SPDEs is of independent interest  and can also be applied to solve many other interesting models such as the stochastic $p$-Laplacian equations, stochastic Allen-Cahn equation and stochastic  heat equations on time-dependent domains or hypersurfaces.  (Monotone) Operator-valued calculus and geometric analysis of moving hypersurfaces play important roles in the study. A forthcoming result on the well-posedness of stochastic 2D Navier-Stokes equation on moving domains is also based on our framework.


	\end{abstract}
	
	
	\medskip
	
	\newpage
	\section{Introduction}
	\ \ \ \ In this paper, we aim to study the well-posedness of  stochastic Stefan problem on moving hypersurfaces. The Stefan problem (\ref{sfe}) models the melting or freezing process of ice-water mixture on a moving hypersurface, with a random heating source.  Stefan problem is one of the most classical and well known free boundary problems. Let $\theta:Q_T\rightarrow\mathbb{R}$ be the temperature function.  For any $t\in[0,T]$ we can decompose the hypersueface $\Gamma_t$ into $\Gamma_t=\Gamma_t^l\cup\Gamma_t^s\cup\Gamma_t^0$ such that
	\begin{align*}
		&\theta(t,x)>0\ \text{in}\ \Gamma_t^l,\ \theta(t,x)=0\ \text{in}\ \Gamma_t^0\ \text{and}\ \theta(t,x)<0\ \text{in}\ \Gamma_t^s.
	\end{align*}
	Similar to (\ref{09022055}), we define
	\begin{align*}
		Q^l_T=\mathop{\scalebox{1.5}[1.5]{$\cup$}}_{t\in[0,T]} \{t\}\times\Gamma^l_t,\ \
		Q^s_T=\mathop{\scalebox{1.5}[1.5]{$\cup$}}_{t\in[0,T]} \{t\}\times\Gamma^s_t\ \text{and}\ 		Q^0_T=\mathop{\scalebox{1.5}[1.5]{$\cup$}}_{t\in[0,T]} \{t\}\times\Gamma^0_t.
	\end{align*}
Then, $\theta$ satisfies the following stochastic partial differential equation on the moving surfaces:
	\begin{numcases}{}\label{sfe-1}
		\nonumber\partial^\bullet \theta(t,x)+\big(\theta(t,x)+1\big)\nabla_{\Gamma_t}\cdot v(t,x)=\Delta_{\Gamma_t}\theta(t,x)+B\big(t,\theta(t,\cdot)\big)(x)\partial^\bullet W_t,\ \text{in}\ Q^l_T,
		\\
		\nonumber\partial^\bullet \theta(t,x)+\theta(t,x)\nabla_{\Gamma_t}\cdot v(t,x)=\Delta_{\Gamma_t}\theta(t,x)+B\big(t,\theta(t,\cdot)\big)(x)\partial^\bullet W_t,\ \text{in}\ Q^s_T,\\
		-\big(\nabla_{\Gamma_t}\theta^l(t,x)-\nabla_{\Gamma_t}\theta^s(t,x)\big)\cdot\mu=V\ \text{on}\ Q_T^0,\\
		\nonumber\theta(t,x)=0\ \text{on}\ Q_T^0,\\
		\theta(0)=x_0\ \text{on}\ \Gamma_0.\nonumber
	\end{numcases}
	In the formulation above, $\mu(t)$ denotes the unit conormal vector on $\Gamma_t^0$ pointing into $\Gamma_t^l$ , $V(t)$ is the conormal velocity of $\Gamma_t^0$ and $v(t)$ is the velocity of $\Gamma_t$.
\vskip 0.2cm
With the transformation $\theta=\beta(X)$ (in the case $a=b=\rho=1$), one can show that the random field $X(t,x)$ is a solution to  equation (\ref{sfe}). For the formal derivation of (\ref{sfe}) and (\ref{sfe-1}), we refer the readers to Subsection 5.2 of \cite{DE} and \cite{AE} for the details. We remark that the surface flux $q$ in \cite{DE}, representing the heat flux per unit surface, is taken as $\nabla_{\Gamma_\cdot} \beta(X)$ in the current context. Many papers, e.g. \cite{BD,KM,KMS,KZS}, are devoted to the study of the stochastic Stefan problems on fixed domains. As far as we know, this is the first paper to study this problem on moving domains/surfaces.\par
	
	Domains that evolve over time arise naturally in applications where the object under study changes shape with time. Many physical phenomena, such as thermal expansion and contraction, as well as fluid movement, exhibit this characteristic. It is very important to understand how the object behaves during this shape-changing process. Such behavior often involves an intricate interaction between the equation and the boundary, making it a challenging topic to study.
	\par
	In this paper, we aim to establish the well-posedness of stochastic Stefan problem on time-dependent hypersurfaces driven by multiplicative noise. While there exist some results on the well-posedness of stochastic partial differential equations (SPDEs) on moving domains (a family of time-dependent open sets $\{\mathcal{O}_t\}_{t\in[0,T]}$ ) with additive noise, see e.g. \cite{WZZ,YYK,ZH}, their approach is similar to the deterministic case. When multiplicative noise is considered, uniqueness becomes a challenge. It would rely on an $\mathrm{It\hat{o}}$ formula for SPDEs on time-dependent domains which generally do not exist.
	We mention that in \cite{PWZZ}, a family  $\big\{e_k(t,x), (t,x)\in \mathcal{D}_T\big(\dot{=}\cup_{t\in [0, T]}\{t\}\times \mathcal{O}_t\big) \big\}_{k\geq 1}$ of eigenfunctions of the Laplace operator $\Delta_t$ with Dirichlet boundary on $L^2(\mathcal{O}_t)$, are used to obtain the well-posedness of the solution $u(t,x):\mathcal{D}_T\rightarrow\mathbb{R}$ to the stochastic heat equation with multiplicative noise on a time-dependent domain as well as deriving the $\mathrm{It\hat{o}}$ formula for $\|u(t)\|_{L^2(\mathcal{O}_t)}^2$. However,  this method depends on the very specific shape of the domain, imposing strict assumptions on the time-evolution of the domain.     \par
	To solve the stochastic Stefan problem on moving hypersurfaces, we intend to use a specially designed transformation and solve the eqaution obtained by transformation. This method has been extensively used in the deterministic contexts, see e.g. \cite{MT,K,HH,KRS,FKW,KMNW} for details. In this paper ,we first establish a correspondence between the solutions of the Stefan problem on moving surfaces and the solutions of the transformed stochastic partial differential equations (SPDEs) on a fixed domain. This is done by obtaining  the correspondence between the classes of test functions in the weak formulation of Stefen probelm on moving hypersurfaces and the weak formulation of transformed SPDEs on a fixed domain. The next goal is to obtain the well-posedness of the transformed SPDEs on the fixed domain. However, it turns out that the coefficients of the transformed SPDE do not satisfy the usual (fully) local monotonicity conditions imposed in the existing literature, see \cite{LR, LR2, NTT, RSZ}. Instead, the transformed SPDE satisfies a kind of nonhomogeneous monotonicity which we now briefly describe.

	Consider SPDEs under the variational setting in a Gelfand triple $V\hookrightarrow H\ \widetilde{=}\ H^*\hookrightarrow V^*$ written as
	\begin{numcases}{}\label{spde}
		dX_t=A(t,X_t)dt+ B(t,X_t)dW_t,\ t\in (0,T],\\
		X_0=\xi\in H,\nonumber
	\end{numcases}
	where $V$ is a Banach space continuously embedded into a Hilbert space $H$,
	$$A: [0,T] \times V\rightarrow V^*,\ \ B:[0,T]\times V\rightarrow\ L_2(U,H)$$ are measurable maps, $L_2(U,H)$ is the space of Hilbert-Schmidt operators from a separable Hilbert space $U$ to $H$ and $W$ is an $U$-cylindrical Wiener process. The operator $A(\cdot,\cdot)$, $B(\cdot,\cdot)$ satisfy the nonhomogeneous monotonicity (see Subsection 2.2 for more details):
	
	There exists a locally bounded measurable function $\rho:V\rightarrow\mathbb{R}$, nonnegative constants $\gamma$ and $C$ such that for $ a.e.$ $t\in[0,T]$, the following inequalities hold for any $u,v\in V$,
	\begin{align}
		\nonumber&2\big\langle \iota_t^*A(t,u)-\iota_t^*A(t,v), u-v\big\rangle+\|B(t,u)-B(t,v)\|^2_{L_2(U,H^t)}\\\label{08280034}\leq
		&\big[f(t)+\rho(v)\big]|u-v|_t^2,\\&
		|\rho(u)|\leq C(1+\|u\|_V^\alpha)(1+|u|^\gamma).\nonumber
	\end{align}
	Here, $\{\iota_t^*\}_t\geq 0$ is a family of symmetric bounded operators from $H$ into $H$, associated with a family of inner products $\{(\cdot,\cdot)_t\}_{t\in[0,T]}$ to be specialized in Subsection 2.2. We now give an interpretation for the nonhomogeneous monotonicity. Let $\{(\cdot,\cdot)_t\}_{t\in[0,T]}$ be a family of inner products equivalent to the original inner product $(\cdot,\cdot)$ on $H$. Then, we denote $H^t$ the Hilbert space $H$ equipped with the time-dependent  inner product $(\cdot,\cdot)_t$.
	Roughly speaking, nonhomogeneous monotonicity refers to  monotonicity with regard to the time-dependent Gelfand triple $V\hookrightarrow H^t\hookrightarrow V^*$ for each $t\in[0,T]$. It is a transformed version of the family of the time-dependent  Gelfand triple $V(t)\hookrightarrow H(t)\hookrightarrow V(t)^*$ for SPDEs on moving domains/hypersurfaces.
	\par Therefore, the first objective of this paper is to solve SPDEs with nonhomogeneous monotonicity, which is of independent interest itself. The well-posedness of SPDEs with monotone operator was first established by Pardoux in \cite{P}, and has since  been widely studied in the literature. We refer the readers to  \cite{LR, LR2, NTT, RSZ} and references therein for subsequent contributions on this topic. To solve SPDEs with nonhomogeneous monotonicity,  the conventional Galerkin approximating scheme becomes ineffective. Instead, we consider a family of time-dependent projections $\{P_n(t):H\rightarrow H\}_{n\geq 1}$, which are orthogonal projections onto finite-dimensional subspaces with respect to the inner product $(\cdot,\cdot)_t$ for each $t\in[0,T]$, and construct a new family of approximating solutions. Since $\{P_n(t):H\rightarrow H\}_{n\geq 1}$ are no longer self-adjoint under the original inner product $(\cdot,\cdot)$, many proofs for the convergence of the approximating solutions in the case of usual monotonicity do not work. A particular role in our proofs is played by the modified version of $\mathrm{It\hat{o}}$ formula we proved for the time-dependent norms $|\cdot|_t^2$ of a $V^*$-valued semimartingale.
	
	\par
	The next part is to  apply the new framework of nonhomogeneous monotonicity to the study of stochastic Stefan problem. We like to stress that the methodology used in \cite{AE,ACDE} for solving  the deterministic Stefan problem on moving hypersurfaces falls short in the stochastic case since their proofs heavily rely on the differentiability of the solution with respect to  time. Through a designed transformation, we introduced a family of inner product $(\cdot, \cdot)_t$ on a suitable Hilbert space and obtained  the necessary properties required in the framework of nonhomogeneous monotonicity. The proofs involve exploiting the geometric structure of the hypersurfaces and the uniform in time regularity estimate for the solutions of  a family of elliptic equations on hypersurfaces.
	We next give a  proof for the correspondence between the solutions of stochastic Stefan problem on the moving hypersurfaces  and the solution of the transformed SPDEs on a fixed domain. As a final step to obtain the well-posedness of stochastic Stefan problem, we proved that the transformed SPDE falls into the framework of nonhomogeneous monotonicity satisfying all the assumptions needed.
	
	Finally we like to mention  that the new framework of nonhomogeneous monotonicity  we proposed can also be applied to solve many other interesting models such as the stochastic $p$-Laplacian equations, one-dimensional Burgers equation, one-dimensional Allen-Cahn equation and the strong solution of the heat equations on a time-dependent domain or hypersurfaces etc. In particular, we would like to mention that based on our framework, we successfully solve the two dimensional stochastic Navier-Stokes equation on moving domains. This result is contained in a forthcoming preprint.
	\vskip 0.3cm
	\par
	The remaining part of the paper are organized  as follows. Subsection 2 is devoted to the well-posedness of SPDEs with nonhomogeneous monotonicity.  We  introduce the assumptions on the family of  inner products $(\cdot,\cdot)_t$ and establish $\mathrm{It\hat{o}}$'s formula for the time-dependent norms $|\cdot|_t^2$ of a semimartingale. We present the precise  assumptions on the coefficients in the framework of nonhomogeneous monotonicity, and prove the existence and uniqueness of the solutions of the SPDEs in this new framework.  In Section 3, we study the stochastic Stefan problem on moving hypersurfaces,
	establish the equivalence of the stochastic Stefan problem to an SPDE with nonhomogeneous monotonicity on a fixed domain and verify all the conditions required to obtain the well-posedness of the problem. Finally, we give an $\mathrm{It\hat{o}}$ formula for the norm of the solution. \par
	\vskip 0.3cm
	Here are some conventions used throughout this paper: \par
	(i)$C$
	denotes a generic positive constant whose value may vary from line to line.
	Other constants will be denoted as $C_1$, $C_2$,$\cdots$. They are all positive, but
	their values are not important. The dependence of constants on parameters,
	if necessary, will be indicated, e.g. $C_T$ .\par
	(ii) We will denote $\mathcal{L}(X,Y)$ to be the space of bounded linear operators from $X$ to $Y$, with its norm denoted by $\|\cdot\|_{\mathcal{L}(X,Y)}$. If $X=Y$, we will use $\mathcal{L}(X)$ instead of $\mathcal{L}(X,Y)$, with its norm denoted by $\|\cdot\|_{\mathcal{L}(X)}$.\par
	(iii)  We will sometimes use Einstein summation convention for simplicity.
	
	\section{SPDEs with non-homogeneous monotonicity}
	\ \ \ \ This section introduces a novel framework to establish the well-posedness of SPDEs with the so called non-homogeneous monotonicity, which will be explained in Subsection 2.2 in detail. This new framework is designed mainly to solve SPDEs on moving surfaces/time-dependent domains.
	
	\subsection{A motivating example}
	\qquad To give the motivation for the introduction of the new framework of nonhomogeneous monotonicity, we consider the heat equation on time-dependent domains. The discussions in this subsection are carried out by formal calculations.  Rigorous arguments will be presented in Section 3. Consider the heat equation on a time-dependent domain $\mathcal{D}_T$:
	\begin{numcases}{}\label{he}
		\nonumber\frac{\partial u}{\partial t}(t,x)=\Delta_t u(t,x),\ (t,x)\in \mathcal{D}_T,\\
		u(t,x)=0,\ (t,x)\in \partial \mathcal{D}_T,\\
		u(0,x)=u_0(x),\ x\in\mathcal{O}_0.\nonumber
	\end{numcases}
	Where  the non-cylindrical domain $\mathcal{D}_T$ is written as
	\begin{align}
		\mathcal{D}_T=\ \mathop{\scalebox{1.5}[1.5]{$\cup$}}_{t\in[0,T]} \{t\}\times\mathcal{O}_t,
	\end{align}
	where for any $t\in[0,T]$, $\mathcal{O}_t$ is a bounded open set in $\mathbb{R}^d$, $\Delta_t$ is the Laplace operator on the domain $\mathcal{O}_t$. Moreover, we assume that there exists a map $r\in C^1\big([0,T]\times\overline{\mathcal{O}}_0,\mathbb{R}^d\big)$ such that
	\begin{itemize}
		\item [(i)] For any $t\in[0,T]$ fixed, $r(t,\cdot)$ maps $\mathcal{O}_0$ onto $\mathcal{O}_t$. Moreover, $r(t,\cdot):\mathcal{O}_0\rightarrow\mathcal{O}_t$ is a $C^2$ diffeomorphism.
		\begin{spacing}{-1.0}$\text{ }$ \end{spacing}
		\item [(ii)] For any $i,j,k\in\{1,2,...,d\}$, $\frac{\partial r_i}{\partial t}$, $\frac{\partial r_i}{\partial y_j}$, $\frac{\partial^2 r_i}{\partial t\partial y_j}$  and $\frac{\partial^2 r_i}{\partial y_j\partial y_k}\in C([0,T]\times\bar{\mathcal{O}}_0,\mathbb{R})$.
	\end{itemize}
	\begin{remark}\label{04261800}
		Let's denote by $\bar{r}(t,\cdot):\mathcal{O}_t\rightarrow\mathcal{O}_0$, the inverse map of $r(t,\cdot)$. By our assumptions  (i), (ii) and the chain rule,  we  easily see that
		\begin{itemize}
			\item [(iii)] For any $i,j,k\in\{1,2,...,d\}$, $\frac{\partial \bar{r}_i}{\partial t}$, $\frac{\partial \bar{r}_i}{\partial x_j}$, $\frac{\partial^2 \bar{r}_i}{\partial t\partial x_j}$ and $\frac{\partial^2 \bar{r}_i}{\partial x_j\partial x_k}\in C(\overline{\mathcal{D}}_T,\mathbb{R})$.
		\end{itemize}
	\end{remark}
	For a solution $u:\mathcal{D}_T\rightarrow\mathbb{R}$ satisfying (\ref{he}), define $v(\cdot,\cdot):[0,T]\times\mathcal{O}_0\rightarrow \mathbb{R}$ by $v(t,y)=u\big(t,r(t,y)\big)$.
	Changing the coordinate by the space-time diffeomorphism
	$$ (t,x)\in \mathcal{D}_T\rightarrow (t,y)\dot{=}\big(t,\bar{r}(t,x)\big)\in [0,T]\times\mathcal{O}_0 $$
	and using the chain rule, we can show that $v(t,y)$ satisfies the following equation on the cylindrical domain $[0,T]\times\mathcal{O}_0$,
	\begin{numcases}{}\label{he2}
		\frac{\partial v}{\partial t}(t,y)=\frac{\partial}{\partial y_j}\big( a_{jk}(t,y)\frac{\partial v}{\partial y_k}(t,y)\big)+(b^1(t,y)+b^2(t,y))\cdot\nabla_y v(t,y)\nonumber,\\  \hfill(t,y)\in(0,T]\times\mathcal{O}_0,\nonumber\\
		v(t,y)=0,\ (t,y)\in (0,T]\times\partial{\mathcal{O}_0},\nonumber\\
		v(0,y)=v_0(y),\ y\in\mathcal{O}_0,
	\end{numcases}
	where $$a_{jk}(t,y)=\sum_{i=1}^{d}\frac{\partial \bar{r}_j}{\partial x_i}(t,r(t,y))\frac{\partial \bar{r}_k}{\partial x_i}(t,r(t,y)),$$ \ $$\ \  b^1_k(t,y)=-\sum_{j=1}^{d}\frac{\partial a_{jk}}{\partial y_j}(t,y)+\Delta_x \bar{r}_k(t,r(t,y))\ \text{and}\  b^2_k(t,y)=\frac{\partial \bar{r}_k}{\partial t}(t,r(t,y)).$$
	Observe that if we define the operators:
	$$A_i(\cdot,\cdot):[0,T]\times H_0^1(\mathcal{O}_0)\rightarrow H^{-1}(\mathcal{O}_0),\ i=1,2,$$
	by
	$$\langle A_1(t,v),w\rangle=-\int_{\mathcal{O}_0} a_{jk}(t,y)\frac{\partial v}{\partial y_k}(y)\frac{\partial w}{\partial y_j}(y)dy+\int_{\mathcal{O}_0}b^1(t,y)\cdot\nabla_yv(y) w(y)dy,$$
	$$\langle A_2(t,v),w \rangle=\int_{\mathcal{O}_0} b^2(t,y)\cdot\nabla_y v(y) w(y) dy. $$
	then when $v_0\in L^2(\mathcal{O}_0)$, (\ref{he2}) can be written on a Gelfand triplet $H_0^1(\mathcal{O}_0)\hookrightarrow L^2(\mathcal{O}_0)\hookrightarrow H^{-1}(\mathcal{O}_0)  $  in the variational setting, given by
	\begin{align}\label{04262129}
		dv(t)=A_1(t,v(t))dt+A_2(t,v(t))dt,\ v(0)=v_0.
	\end{align}
	Note that
	\begin{align}\label{04262111}
		-\int_{\mathcal{O}_t} \nabla \bar{v}^t(x)\cdot\nabla \bar{w}^t(x) dx= _{H^{-1}(\mathcal{O}_0)}\langle A_1(t,v), \big(w|\mathrm{det}(\frac{\partial r_i}{\partial y_j})|(t,\cdot)\big)\rangle_{H_0^1(\mathcal{O}_0)},
	\end{align}
	for any $v,w\in H_0^1(\mathcal{O}_0)$, where  $\bar{v}^t(x)=v\big(\bar{r}(t,x)\big)$, $\bar{w}^t(x)=w\big(\bar{r}(t,x)\big)$ and $(w|det(\frac{\partial r_i}{\partial y_j})|(t,\cdot))(y)=w(y)|det\big(\frac{\partial r_i}{\partial y_j}\big)|(t,y)$. It follows from (\ref{04262111}) that it is the operator $|det\big(\frac{\partial r_i}{\partial y_j}\big)(t,\cdot)|A_1(t,\cdot) $ that inherits the monotonicity of $-\Delta_t$ on the Gelfand triple $H_0^1(\mathcal{O}_t)\hookrightarrow L^2(\mathcal{O}_t)\hookrightarrow H^{-1}(\mathcal{O}_t)$, rather than the operator $A_1(t,\cdot) $. $A_1(t,\cdot)$ is the operator with the so called nonhomogeneous monotonicity. \par
	
	In the next subsection, we will lay the groundwork for introducing nonhomogeneous monotonicity by establishing a family of equivalent inner products on a Hilbert space, which are  compatible with the inner products on the time-dependent Hilbert spaces associated with  the function spaces on time-dependent domains. In the next subsection, we will give some preliminaries for the introduction on the concept ``nonhomogeneous monotonicity". 
	\subsection{Assumptions on the inner products}
	\ \ \ \ Let $H$ be a separable Hilbert space with inner product $( \cdot,\cdot )$ and norm $|\cdot|$. Let $V$ be a reflexive Banach space that is continuously and densely embedded into $H$ and let $V^*$ be the dual space of $V$. We denote the norms of $V$ and $V^*$ by $\|\cdot\|_V$ and $\|\cdot\|_{V^*}$ respectively. By the Riesz representation, the Hilbert space $H$ can be identified with its dual space $H^*$. Thus we obtain a Gelfand triple
	$$V\subseteq H\subseteq V^*.$$
	We denote the dual pair between $f\in V^*$ and $v\in V$ by $\langle f,v \rangle$. It is easy to see that when $u\in H\subseteq V^*$, $(u,v)=\langle u,v\rangle$, $\forall\   v\in V$.\par
	Assuming that there exists a family of inner products $(\cdot,\cdot)_t$ on $H$ indexed by $t\in[0,T]$, such that for $t=0$, the inner product $(\cdot,\cdot)_0$ coincides with the original inner product $(\cdot,\cdot)$ on $H$. Moreover, we assume that $(\cdot,\cdot)_t, t\in[0,T]$ satisfies the following conditions:
	\begin{itemize}
		\item[\hypertarget{C1}{{\bf (C1)}}] There exists a constant $c_1\geq 1$ such that for any $t\in[0,T]$, the norm $|\cdot|_t$ generated by the inner product $(\cdot,\cdot)_t$ satisfies,
		\begin{align*}
			\frac{1}{c_1}|x|\leq |x|_t\leq c_1|x|,\text{ for any $x\in H$}. 
		\end{align*}
		\item[\hypertarget{C2}{{\bf (C2)}}] There exists a family of self-adjoint bounded operators $\{\Phi(t)\}_{t\in[0,T]}$ on $H$ such that
		\begin{align*}
			\int_{0}^{T}\|\Phi(t)\|_{\mathcal{L}(H)}dt<\infty\  ,
		\end{align*}
		and for any $t\in[0,T]$,
		\begin{align}\label{08201354}
			|x|_t^2-|x|^2= \int_{0}^{t}(x,\Phi(s)x)ds. 
		\end{align}
	\end{itemize}
	
	\begin{remark}
		We cannot deduce that $\Phi(s):H\rightarrow H$ is  self-adjoint merely from (\ref{08201354}). However, if we replace $\Phi(s)$ by $\frac{1}{2}\big(\Phi(s)+\Phi^*(s)\big)$, (\hyperlink{C2}{C2}) is satisfied.
	\end{remark}
	
	To state other conditions, we need to introduce some notations. By condition (\hyperlink{C1}{C1}) and the Lax-Milgram theorem, for every $t\in[0,T]$, there exists a self-adjoint bounded operator $\iota_t^*:H\rightarrow H$ such that $(\iota_t^* x,y)=(x,y)_t$ for any $x,y\in H$. And it is easy to see that $\|\iota_t^*\|_{\mathcal{L}(H)}\leq c_1^2$.
	The following lemma follows easily from (\hyperlink{C1}{C1}) and (\hyperlink{C2}{C2}).
	\begin{lemma}\label{01201913}
		For the operators $\iota_t^*$ and $\Phi(t)$ on $H$ defined above, we have\\
		(i) For any $t\in[0,T]$, $x\in H$, $\iota_t^*x=x+\int_{0}^{t}\Phi(s)xds$.\\
		(ii) For any $s<t$ in $[0,T]$, $\|\iota_t^*-\iota_s^*\|_{\mathcal{L}(H)}\leq \int_{s}^{t}\|\Phi(r)\|_{\mathcal{L}(H)}dr$.\\
		(iii) For any $t\in[0,T]$, $\iota_t^*$ is invertible on $H$, with its inverse operator denoted by $\iota_{-t}^*$, satisfying $\|\iota_{-t}^*\|_{\mathcal{L}(H)}\leq c_1^2$ for any $t\in[0,T]$. Here $c_1>0$ is  the constant introduced in (\hyperlink{C1}{C1}).
	\end{lemma}
	\noindent {\bf Proof}. Clearly, $(ii)$ follows from $(i)$. Noticing that $(x, y)_t=\frac{1}{4}(|x+y|_t^2-|x-y|_t^2)$, we obtain from
	(\hyperlink{C2}{C2}) that
	\begin{align}\label{01210231}
		(x,y)_t=(\iota_t^*x,y)=(x,y)+\int_{0}^{t}(\Phi(s)x,y)ds=\big(x+\int_{0}^{t}\Phi(s)xds,y\big),
	\end{align}
	for any $t\in[0,T]$ and $y\in H$. Thus $\iota_t^*x=x+\int_{0}^{t}\Phi(s)xds$ and $(i)$ is proved. For the proof of (iii), we use the Lax-Milgram theorem and (\hyperlink{C1}{C1}) to obtain that $\iota_t^*$ is invertible and $\|(\iota_{-t}^*)\|_{\mathcal{L}(H)}\leq c_1^2$.$\hfill\blacksquare$\par
	\begin{remark}\label{01210304}
		It follows from the formula (\ref{01210231}) that for any $x,y\in H$, the function $t\rightarrow(x,y)_t$ is continuous on $[0,T]$.
	\end{remark}
	
	In the sequel, we will write the linear space $H$ equipped with the inner product $(\cdot,\cdot)_t$ as $H^t$. We further assume that
	\begin{itemize}
		\item[(C3)]
		the operator $\iota_t^*$ is  a bounded operator on $V$ as well, and there exists a constant $c_2>0$ such that
		\begin{align*}
			\|\iota_t^* x\|_V\leq c_2\|x\|_V \text{ for any $x\in V,  t\in[0,T]$}. 
		\end{align*}
		\item[(C4)] $\iota_t^*$ is bijective from $V$ to $V$ and there exists a constant $c_3>0$ such that
		\begin{align*}
			\text{  $\|\iota_{-t}^*x\|_{V}\leq c_3\|x\|_{V}$ for any $x\in V$, $t\in[0,T]$.} 
		\end{align*}
	\end{itemize}
	Under (\hyperlink{C3}{C3}), for any $t\in[0,T],\ f\in V^*$, we  define $\iota_t^*f$ to be the element in $V^*$ such that
	$\langle \iota_t^*f,v\rangle=\langle f,\iota_t^*v\rangle$ for any $v\in V$. This coincides with $\iota_t^*f\in H$ when $f\in H\subseteq\ V^*$. And
	we remark that (\hyperlink{C3}{C3}) and
	(\hyperlink{C4}{C4}) imply that
	for any $f\in V^*$ and $t\in[0,T]$,
	\begin{eqnarray}\label{eq 20240720 00}
		\frac{1}{c_3}\|f\|_{V^*}
		\leq
		\|\iota_{t}^*f\|_{V^*}
		\leq
		c_2 \|f\|_{V^*}.
	\end{eqnarray}

	We now derive the following $\mathrm{It\hat{o}}$ formula for a $V^*$-valued semimartingale from our assumptions. Let $(\Omega,\mathcal{F},\mathcal{F}_t,\mathbb{P})$ be a complete filtered probability space satisfying the usual conditions. Let $\{W_t, t\geq 0\}$ be a $U$-cylindrical Brownian motion on this probability space, where $U$ is a Hilbert space.
	
	\begin{theorem}\label{ito}
		Assuming (\hyperlink{C1}{C1})-(\hyperlink{C3}{C3}).  Let $\alpha\in(1,+\infty)$, $X_0\in L^2(\Omega,\mathcal{F}_0,\mathbb{P},H)$, $Y\in L^\frac{\alpha}{\alpha-1}([0,T]\times \Omega, dt\otimes \mathbb{P},V^*)$ and $Z\in L^2\big([0,T]\times \Omega, dt\otimes \mathbb{P},L_2(U,H)\big)$ both progressively measurable. If $X$ is a $V^*$-valued continuous process of the form,
		\begin{align}\label{01202030}
			X_t=X_0+\int_{0}^{t}Y_s ds+\int_{0}^{t}Z_sdW_s,\ \forall t\in[0,T]
		\end{align}
		and $X\in L^\alpha([0,T]\times\Omega, dt\otimes\mathbb{P},V)$. Then $X_\cdot$ has an $H$-continuous version and for any $t\in [0,T]$, the following equality holds $\mathbb{P}-a.s.$
		\begin{align*}
			\nonumber|X_t|_t^2=\ &|X_0|^2+\int_{0}^{t}\big(2\langle Y_s, \iota_s^* X_s\rangle+\|Z_s\|^2_{L_2(U,H^s)}\big)ds+\int_{0}^{t}\big( X_s,\Phi(s)X_s \big)ds\\&+2\int_{0}^{t}\big(X_s,Z_s dW_s\big)_s.
		\end{align*}
		\begin{remark}
			In the equation above, we use the notation $\int_{0}^{t}\big(X_s,Z_s dW_s\big)_s$ to represent the stochastic integral $\int_{0}^{t}\big(X_s,\iota_s^*Z_s dW_s\big)$, where $\iota^*_sZ_s\in L_2(U,H^s)$ is the composition of the operator $\iota^*_s$ and $Z_s$ for any $s\in[0,T]$. \par
			We give an alternative interpretation of this term, which motivates the adoption of this notation. Indeed, take an orthonormal basis $\{f_i\}_{i\geq 1}$ of the Hilbert space $U$ and a sequence of real-valued independent Brownian motion $\{\beta^i\}_{i\geq 1}$ such that the $U$-cylindrical Brownian motion $W$ can be expressed as
			\begin{align}\label{01201852}
				dW_t=\sum_{i=1}^{\infty}f_id\beta^i_t.
			\end{align}
			Then, we have $$\int_{0}^{t}\big(X_s,\iota_s^*Z_s dW_s\big)=\sum_{i=1}^{\infty}\int_{0}^{t}\big(X_s,\iota_s^*Z_s  {f_i}\big)d{\beta^i_s}=\sum_{i=1}^{\infty}\int_{0}^{t}\big(X_s,Z_s f_i\big)_sd{\beta^i_s}.$$
			The expression on the right-hand-side motivates the use of the notation $\int_{0}^{t}\big(X_s,Z_s dW_s\big)_s$.
		\end{remark}
	\end{theorem}
	\noindent {\bf Proof of Theorem \ref{ito}}. The $H$-continuous version of $X$ follows from Theorem 4.2.5 in \cite{LR}. By (i) of Lemma \ref{01201913}, for any $e\in V$ and $t\in[0,T]$, $$\iota_t^*e=e+\int_{0}^{t}\Phi(s)eds.$$ By (\hyperlink{C3}{C3}), we observe that $\iota_\cdot^*e$ belongs to $L^\alpha([0,T],V)$, with its derivative $\Phi(\cdot)e\in L^1([0,T],H)$. Thus, we can apply the $\mathrm{It\hat{o}}$ formula (see e.g Theorem 3.1 in Section 2 of \cite{P}) to obtain
	\begin{align*}
		\big(X_t,\iota_t^* e\big)=\big(X_0, e\big)+\int_{0}^{t}\langle Y_s, \iota_s^*e \rangle ds+\int_{0}^{t}(Z_sdW_s,\iota_s^*e)+\int_{0}^{t}(\Phi(s)e,X_s)ds.
	\end{align*}
	Therefore, the following equation holds $\mathbb{P}-a.s$ for any $e\in V$ and $t\in[0,T]$,
	\begin{align*}
		\langle \iota_t^*X_t,e\rangle=\langle X_0,e\rangle+\int_{0}^{t}\langle\iota_s^*Y_s,e \rangle ds+\int_{0}^{t}(\iota_s^*Z_sdW_s,e)+\int_{0}^{t}(\Phi(s)X_s,e)ds.
	\end{align*}
	Thus
	\begin{align}\label{01202028}
		\iota_t^*X_t=X_0+\int_{0}^{t}\iota_s^*Y_s ds+\int_{0}^{t}\iota_s^*Z_s dW_s+\int_{0}^{t}\Phi(s)X_sds.
	\end{align}
	Applying the product rules to (\ref{01202030}) and (\ref{01202028}), we obtain
	\begin{align*}
		&(X_t,\iota_t^*X_t)\\=\ & (X_0,X_0)+\int_{0}^{t}\langle\iota_s^*Y_s,X_s\rangle ds+\int_{0}^{t}(\iota_s^*Z_sdW_s,X_s)+\int_{0}^{t}(\Phi(s)X_s,X_s)ds\\ &+\ \int_{0}^{t}\langle Y_s,\iota_s^*X_s \rangle ds+\int_{0}^{t}\big( Z_sdW_s,\iota_s^*X_s \big)+\int_{0}^{t}\langle \iota_s^*Z_s, Z_s\rangle_{L_2(U,H)} ds\\=\ &(X_0,X_0)+2\int_{0}^{t}\langle Y_s,\iota_s^*X_s\rangle ds+\ 2\int_{0}^{t}\big( Z_sdW_s,X_s \big)_s+\int_{0}^{t}\langle \iota_s^*Z_s, Z_s\rangle_{L_2(U,H)} ds\\ &+\int_{0}^{t}(\Phi(s)X_s,X_s)ds.
	\end{align*}
	The definition of $\iota_t^*$ yields that $$(X_t,\iota_t^*X_t)=|X_t|_t^2,$$ and
	\begin{align*}
		\langle \iota_s^*Z_s, Z_s\rangle_{L_2(U,H)}=\sum_{i=1}^{\infty} (Z_sf_i,\iota^*_sZ_sf_i)=\sum_{i=1}^{\infty}(Z_sf_i,Z_sf_i)_{s}=\|Z_s\|_{L_2(U,H^s)}^2,
	\end{align*}
	completing the proof of  Theorem \ref{ito}.$\hfill\blacksquare$\par
	Let's now return to the motivative example given in Subsection 2.1. In that case, $V=H_{0}^1(\mathcal{O}_0)$ and $H=L^2(\mathcal{O}_0)$, the product $(\cdot,\cdot)_t$ on $H$ is given by
	\begin{align*}
		(u,v)_t=\int_{\mathcal{O}_0}u(y)v(y)|det(\frac{\partial r_i}{\partial y_j})|(t,y)dy.
	\end{align*}
	Thus, we have $\iota_t^*u(y)=|\det(\frac{\partial r_i}{\partial y_j})|(t,y)u(y)$ for $u\in H$ and $y\in\mathcal{O}_0$. It then follows from (\ref{04262111}) that for any $v\in V$,
	\begin{align}\label{08201711}
		\langle \iota_t^*A_1(t,v),v\rangle\leq 0.
	\end{align}
	We would like to mention that our assumption on nonhomogeneous monotonicity, given in the next subsection, is a generalized version of (\ref{08201711}).
	\subsection{SPDEs with nonhomogeneous monotonicity}
	\ \ \ \ In this subsection, we will present the results on the well-posedness of  SPDEs with nonhomogeneous monotonicity.
	We focus on stochastic evolution equation (\ref{spde}) with a deterministic initial data $\xi\in H$ in the variational framework. To establish the well-posedness, we impose the following assumptions on the coefficients $A(\cdot,\cdot)$ and $B(\cdot,\cdot)$.
	Let $f\in L^1([0,T],\mathbb{R}_+)$ and $\alpha\in(1,\infty)$. Assume (\hyperlink{C1}{C1})-(\hyperlink{C4}{C4}) are satisfied for the operators  $\iota_t^*, t\geq 0$ defined in Subsection 2.2.
	\vskip -0.5cm
	\begin{itemize}
		\item [\hypertarget{H1}{{\bf (H1)}}] (Hemicontinuity) For $a.e.$
		$t\in[0,T]$, the map $\lambda\in\mathbb{R}\rightarrow\big\langle A(t,u+\lambda v),\iota_t^*x\big\rangle \in\mathbb{R}$ is continuous, for any $u,v,x\in V$.
		\item [\hypertarget{H2}{{\bf (H2)}}] (Nonhomogeneous Monotonicity) There exists a locally bounded measurable function $\rho:V\rightarrow\mathbb{R}$, nonnegative constants $\gamma$ and $C$ such that for $ a.e.$ $t\in[0,T]$, the following inequalities hold for any $u,v\in V$,
		\begin{align*}
			&2\big\langle \iota_t^*A(t,u)-\iota_t^*A(t,v), u-v\big\rangle+\|B(t,u)-B(t,v)\|^2_{L_2(U,H^t)}\\\leq
			&\big[f(t)+\rho(v)\big]|u-v|_t^2,\\&
			|\rho(u)|\leq C(1+\|u\|_V^\alpha)(1+|u|^\gamma).
		\end{align*}
		\item [\hypertarget{H3}{{\bf (H3)}}] (Coercivity) There exists a constant $c>0$ such that for $a.e.$ $t\in[0,T]$, the following inequality holds for any $u\in V$,
		\begin{align*}
			2\langle\iota^*_t A(t,u),u\rangle+\|B(t,u)\|^2_{L_2(U,H^t)}\leq f(t)\big(1+|u|_t^2\big)-c\|u\|^\alpha_V.
		\end{align*}
		\item [\hypertarget{H4}{{\bf (H4)}}] (Growth) There exist nonnegative constants $\beta$ and $C$ such that for $a.e.$ $t\in[0,T]$, we have for any $u\in V$,
		$$\|\iota_t^*A(t,u)\|^\frac{\alpha}{\alpha-1}_{V^*}\leq\big(f(t)+C\|u\|_V^\alpha\big)\big(1+|u|^\beta\big).$$
		\item [\hypertarget{H5}{{\bf (H5)}}] For $a.e.$ $t\in[0,T]$, we have for any $u\in V$,
		$$\|B(t,u)\|^2_{L_2(U,H^t)}\leq f(t)(1+|u|_t^2).$$
	\end{itemize}
	
	Let $\xi\in H$. We give a precise definition of the solution.
	\begin{definition}\label{01252114}
		An $H$-valued continuous $\{\mathcal{F}_t\}$-adapted process $(X_t)_{t\in[0,T]}$ is called a solution to (\ref{spde}) if $X\in L^\alpha\big(\Omega,L^\alpha([0,T],V)\big)\cap L^2\big(\Omega, C([0,T],H)\big)$ with $\alpha>1$ in (\hyperlink{H3}{H3}) and for any $t\in[0,T]$, the following holds $\mathbb{P}\text{-\ }a.s.$  in $V^*$,
		\begin{align*}
			X_t=\xi+\int_{0}^{t}A(s,X_s)ds+\int_{0}^{t}B(s,X_s)dW_s.
		\end{align*}
	\end{definition}

	Now we can state our main result on the well-posedness.
	\begin{theorem}\label{wp}
		Suppose (\hyperlink{H1}{H1})-(\hyperlink{H5}{H5}) and (\hyperlink{C1}{C1})-(\hyperlink{C4}{C4}) hold. Then for any $\xi\in H$, the equation (\ref{spde}) has a unique solution $(X_t)_{t\in[0,T]}$ such that $X_0=\xi$ and $$E\big[\sup_{t\in[0,T]}|X_t|_t^2+\int_{0}^{T}\|X_t\|_V^\alpha dt\big]<\infty.$$
	\end{theorem}
	\begin{remark}	
		We mention that the requirement on the initial $\xi$ can be relaxed to $\xi\in L^p(\Omega,\mathcal{F}_0,H)$ for some $p\geq 2$ depend on $\beta$ as outlined in Theorem 1.1 in \cite{LR}. Similarly, the assumption (\hyperlink{H5}{H5}) on the noise term can also be loosened to be solely governed by the mixture of the norms in $V$ and $H$, as demonstrated in (1.2) of \cite{BLZ}. Moreover, the coefficient $A(t,\cdot)$ and $B(t,\cdot)$ can be assumed to be random, as discussed in \cite{LR2}. However, for the sake of simplicity, we opt for the assumptions presented here, which are sufficient for the examples we are interested in.
	\end{remark}
	\subsection{Proof of Theorem \ref{wp}}
	\ \ \ \ \ In this subsection, we will give a proof of Theorem \ref{wp}. The proof is based on a time-dependent Galerkin approximation and the modified techniques of monotonicity.\par
	We first show the uniqueness, based on the $\mathrm{It\hat{o}}$ formula stated in Theorem \ref{ito}.\par
	Assume $X$ and $Y$ are two solutions with the same initial data $\xi\in H$. For $M>0$, define
	\begin{align}\label{01260035}
		\nonumber\sigma^M\dot{=}\ T&\wedge\inf\Big\{t\geq 0:\ |X_t|_t\geq M\Big\}\\\nonumber&\wedge\inf\Big\{t\geq 0:\int_{0}^{t} \|X_s\|_V^\alpha ds\geq M\Big\}\\\nonumber&\wedge\inf\Big\{t\geq 0:|Y_t|_t\geq M\Big\}\\&\wedge\inf\Big\{t\geq 0:\int_{0}^{t} \|Y_s\|_V^\alpha ds\geq M\Big\}.
	\end{align}
	By Theorem \ref{ito}, we have
	\begin{align*}
		&\ |X_{t\wedge\sigma^M}-Y_{t\wedge\sigma^M}|^2_{{t\wedge\sigma^M}}\\=&\ \int_{0}^{t\wedge\sigma^M}\Big(2\langle A(s,X_s)-A(s,Y_s),\iota_s^*(X_s-Y_s)\rangle +\|B(s,X_s)-B(s,Y_s)\|^2_{L_2(U,H^s)}\Big)ds\\&+\int_{0}^{t\wedge\sigma^M}\Big(\big(B(s,X_s)-B(s,Y_s)\big)dW_s,X_s-Y_s\Big)_s+\int_{0}^{t\wedge\sigma^M}\big(\Phi(s)(X_s-Y_s),X_s-Y_s\big)ds.
	\end{align*}
	Define
	$$\Psi(t)=\int_{0}^{t}\big(f(s)+c_1^2\|\Phi(s)\|_{\mathcal{L}(H)}+\rho(Y_s)\big)ds.$$
	Apply the $\mathrm{It\hat{o}}$ formula to $e^{-\Psi(t\wedge\sigma^M)}|X_{t\wedge\sigma^M}-Y_{t\wedge\sigma^M}|^2_{t\wedge\sigma^M}$, we obtain
	\begin{align*}
		&e^{-\Psi(t\wedge\sigma^M)}|X_{t\wedge\sigma^M}-Y_{t\wedge\sigma^M}|^2_{t\wedge\sigma^M}\\=\ &\ \int_{0}^{t\wedge\sigma^M}e^{-\Psi(s)}\Big(2\langle A(s,X_s)-A(s,Y_s),\iota_s^*(X_s-Y_s)\rangle +\|B(s,X_s)-B(s,Y_s)\|^2_{L_2(U,H^s)}\Big)ds\\&+\int_{0}^{t\wedge\sigma^M}e^{-\Psi(s)}\Big(\big(B(s,X_s)-B(s,Y_s)\big)dW_s,X_s-Y_s\Big)_s\\&+\int_{0}^{t\wedge\sigma^M}e^{-\Psi(s)}\big(\Phi(s)(X_s-Y_s),X_s-Y_s\big)ds-\int_{0}^{t\wedge\sigma^M}\Psi'(s)e^{-\Psi(s)}|X_s-Y_s|_s^2ds\\\leq&\int_{0}^{t\wedge\sigma^M}e^{-\Psi(s)}\Big(f(s)+\rho(Y_s)+c_1^2\|\Phi(s)\|_{\mathcal{L}(H)}-\Psi'(s)\Big)|X_s-Y_s|_s^2ds\\&+\int_{0}^{t\wedge\sigma^M}e^{-\Psi(s)}\Big(\big(B(s,X_s)-B(s,Y_s)\big)dW_s,X_s-Y_s\Big)_s.
	\end{align*}
	Taking expectation at both sides, we can easily obtain
	\begin{align*}
		E\Big[e^{-\Psi(t\wedge\sigma^M)}|X_{t\wedge\sigma^M}-Y_{t\wedge\sigma^M}|^2_{t\wedge\sigma^M}\Big]=0,
	\end{align*}
	which implies $X_{\cdot}$= $Y_{\cdot}$ on $[0,\sigma^M]$. Since $\sigma^M\rightarrow T$ as $M\rightarrow\infty$, we can assert that $X_{\cdot}=Y_{\cdot}$ on $[0,T]$. \par
	We are now left to show the existence of the solution.
	Due to the dense embedding $V\subseteq H$, we can choose a sequence $\{e_i\}_{i=1}^\infty\subseteq V$ such that $\{e_i\}_{i\geq 1}$ is an orthonormal basis of $\big(H,(\cdot,\cdot)\big)$. Denote $H_n= span\{e_1,\cdots,e_n\}$. For every $t\in[0,T]$, we can apply the Gram-Schmidt orthogonalization to $\{e_i\}_{i\geq 1}$ with regard to the inner product $(\cdot,\cdot)_t$, obtaining an orthonormal basis  $\{e_i(t)\}_{i\geq 1}$  in $H^t$.
	Then, we can define the time-dependent finite-dimensional projection $P_n(t):V^*\rightarrow H_n$ given by
	\begin{align}\label{proj}
		P_n(t)u=\sum_{i=1}^{n}\langle u,\iota^*_te_i(t)\rangle e_i(t)=\sum_{i=1}^{n}\langle \iota^*_tu,e_i(t)\rangle e_i(t),
	\end{align}
	for any $u\in V^*$. In particular, if $u\in H$, by the definition of $\iota_t^*$, we can easily obtain that $$P_n(t)u=\sum_{i=1}^{n}\big(u,e_i(t)\big)_te_i(t),$$
	which implies that $P_n(t):H\rightarrow H_n$, is the orthogonal projection onto $H_n$ under the inner product $(\cdot,\cdot)_t$.\par   We would like to mention that  Remark \ref{01210304} and the linear independence of $\{e_i\}_{i\geq 1}$ implies that
	for any $k\geq 1$, the map: $t\in [0, T]\rightarrow e_k(t)\in H$ is continuous. Thus for any $n\geq 1$, there exists a constant $C_n>0$ such that for any $t\in[0,T]$,
	\begin{align}\label{01220115}
		\|P_n(t)\|_{\mathcal{L}(V^*,H_n)}\leq C_n.
	\end{align}
	Moreover, for $f\in V^*,w\in H_n,t\in[0,T]$,
	\begin{align}\label{01212124}
		\nonumber( P_n(t)f,w)_t&=\sum_{i=1}^{n}\langle f,\iota_t^* e_i(t)\rangle(e_i(t),w)_t\\\nonumber&=\langle f,\iota_t^* [\sum_{i=1}^{n}(w,e_i(t))_te_i(t)]\rangle\\&=\langle f,\iota_t^*w\rangle.
	\end{align}
	Set $$W^n(t)=Q_nW_t\dot{=}\sum_{i=1}^{n}\langle W_t,f_i\rangle f_i,$$
	where $Q_n$ is the orthogonal projection from $U$ onto span\{$f_1,\cdots,f_n$\} for $\{f_n\}_{n\geq 1}$ introduced in (\ref{01201852}).
	Then, we consider the following finite-dimensional approximating equation,
	\begin{numcases}{}
		\nonumber dX^n_t=P_n(t)A(t,X^n_t)dt+P_n(t)B(t,X^n_t)dW^n_t,\\
		X_0^n=P_n(0)\xi\ \dot{=}\ \xi^n\in H_n\label{galerkin}.
	\end{numcases}
	For the well-posedness of (\ref{galerkin}), we reformulate it as an equation on $\mathbb{R}^n$. Consider the maps
	$a_\cdot^n(t,x):[0,T]\times \mathbb{R}^n\rightarrow \mathbb{R}^n$ and $b_{\cdot\cdot}^n(t,x):[0,T]\times \mathbb{R}^n\rightarrow \mathbb{R}^{n\times n}$ given by
	\begin{align}
		a^n_i(t,x)=\big( P_n(t)A(t,\sum_{k=1}^{n}x^ke_k), e_i\big),\  b^n_{ij}(t,x)=\big(P_n(t)B(t,\sum_{k=1}^{n}x^ke_k) f_j,e_i\big).\label{01220053}
	\end{align}
	Write $\big\{(X^n_t,e_k)\big\}_{k=1}^n$ as $x_t\in\mathbb{R}^n$ for any $t\in(0,T$], $\big\{(\xi,e_k)\big\}_{k=1}^n$ as $\xi^n\in\mathbb{R}^n$, and $B_t^n=\{\langle W_t,f_i\rangle\}_{i=1}^n$. Then the Galerkin approximating equation (\ref{galerkin}) is equivalent to the following equation on $\mathbb{R}^n$,
	\begin{numcases}{}
		\nonumber dx_t=a^n(t,x_t)dt+b^n(t,x_t)dB^n_t,\ t\in(0,T],\\
		x_0= \xi^n\ \  \ \label{ode}.
	\end{numcases}
	Denote $a^{ij}_t\dot{=}\ \big( e_i, e_j\big)_t$, for any $x=(x^1,...,x^n),y=(y^1,...,y^n)\in\mathbb{R}^n$, we have for  $t\in[0,T]$,
	\begin{align}\label{01212353}
		\nonumber\sum_{i,j=1}^{n}a_t^{ij}a_i^n(t,x)y^j&=\sum_{i,j=1}^{n}\big(P_n(t)A(t,\sum_{k=1}^{n}x^ke_k),e_i\big)(e_j,e_i)_ty^j\\&\nonumber=\Big(\sum_{i=1}^{n}\big(P_n(t)A(t,\sum_{k=1}^{n}x^ke_k), e_i\big)e_i,\sum_{j=1}^{n}y^je_j\Big)_t\\\nonumber&=\big(P_n(t)A(t,\sum_{k=1}^{n}x^ke_k),\sum_{j=1}^{n}y^je_j\big)_t\\&=\langle A(t,\sum_{k=1}^{n}x^ke_k),\iota_t^*\big(\sum_{j=1}^{n}y^je_j\big)\rangle,
	\end{align}
	where the last equality follows from (\ref{01212124}). Likewise, we also have that for any $t\in[0,T],x,y\in \mathbb{R}^n$,
	\begin{align}\label{01212354}
		\sum_{i,j,k=1}^{n}a^{ij}_tb^n_{ik}(t,x)b^n_{jk}(t,y)=\langle P_n(t)B(t,\sum_{l=1}^{n}x^le_l)Q_n, P_n(t)B(t,\sum_{l=1}^{n}y^le_l)Q_n \rangle_{L_2(U,H^t)}.
	\end{align}
	Then, for the well-posedness of (\ref{ode}), the following lemma is needed.
	\begin{lemma}
		Assume that there exists a family of inner products $( \cdot,\cdot )_t$ on $\mathbb{R}^n$, such that (\hyperlink{C1}{C1}) and (\hyperlink{C2}{C2}) hold with $H{=}\ \mathbb{R}^d$. For the canonical basis $\{e_k\}_{k=1}^d$ of $\mathbb{R}^d$ and $\sigma\in\mathbb{R}^{d\times d}$, define the following norm $\|\cdot\|_t$ on $\mathbb{R}^{d\times d}$ by
		$$\|\sigma\|^2_t=\sum_{k=1}^{d}(\sigma e_k,\sigma e_k)_t	, $$
		which is equivalent to the Frobenius norm.
		Assume that $\sigma(t,x):[0,T]\times\mathbb{R}^d\rightarrow\mathbb{R}^{d\times d},\ b(t,x):[0,T]\times\mathbb{R}^d\rightarrow \mathbb{R}^d$ are measurable and continuous in $x$ for $a.e.$ $t\in[0,T]$ fixed and satisfy for any $R>0$,
		\begin{align}
			\int_{0}^{T}\sup_{|x|\leq R}\Big\{ \|\sigma(t,x)\|_t^2+|b(t,x)|\Big\}dt<\infty. \label{01220047}
		\end{align}
		Moreover, we assume that for any $t\in[0,T],R\in [0,\infty)$, $x,y\in\mathbb{R}^d,\ |x|\vee|y|\leq R$,
		\begin{itemize}{}
			\item[(i)] $2\big( x-y, b(t,x)-b(t,y)\big)_t+\|\sigma(t,x)-\sigma(t,y)\|_t^2\leq K_t(R)|x-y|_t^2,$
			\item[(ii)] $2\big( x,b(t,x) \big)_t+\|\sigma(t,x)\|_t^2\leq K_t(1)(1+|x|_t^2)$,
		\end{itemize}
		where $\forall\ R>0$, $K_t(R)$ is an $\mathbb{R}_+$-valued measurable function satisfying $\int_{0}^{T}K_t(R)dt<\infty$. Then, $\forall\ {\mathcal{F}_0}$-measurable initial $\xi:\Omega\rightarrow\mathbb{R}^d$, there exists an unique solution to the SDE driven by the $d$-dimensional standard Brownian motion $B_.$,
		$$ dX_t=b(t,X_t)dt+\sigma(t,X_t)dB_t,\ X_0=\xi.$$
		\noindent {\bf Proof}. The proof is almost identical to the proof of Theorem 3.1 in \cite{LR2}, using the $\mathrm{It\hat{o}}$ formula we obtained in Theorem \ref{ito}. Thus we omit it here.$\hfill\blacksquare$
	\end{lemma}
	The well-posedness of (\ref{galerkin}) follows from the following corollary.
	\begin{corollary}
		Under the assumptions (\hyperlink{H1}{H1})-(\hyperlink{H5}{H5}) and (\hyperlink{C1}{C1})-(\hyperlink{C2}{C2}). The equation (\ref{ode}) has a unique solution.
	\end{corollary}
	\noindent {\bf Proof}. It suffices to verify the conditions in the lemma above with $b$ and $\sigma$ respectively replaced by $a^n$ and $b^n$ defined in (\ref{01220053}) and $(x,y)_t=a_t^{ij}x^iy^j$. The verification of $(i)$, $(ii)$ and (\ref{01220047}) can be readily inferred from (\hyperlink{H2}{H2})-(\hyperlink{H5}{H5}), (\ref{01212353}), (\ref{01212354}) and  (\ref{01220115}). Subsequently, it suffices to show the continuity in $x\in\mathbb{R}^n$ for $a.e.\ t\in[0,T]$ fixed. The assumptions $(\hyperlink{H2}{H2})$ and $(\hyperlink{H4}{H4})$ imply the continuity for $b^n$. For the continuity of $a^n$, we mention that $(\hyperlink{H1}{H1})$ and $(\hyperlink{H2}{H2})$ implies $\iota_t^*A(t,x)$ is demicontinuous for $a.e.$ $t\in[0,T]$, as proved in Remark 4.1.1 of \cite{LR2}. Subsequently, the continuity follows easily from (\ref{proj}) and  (\ref{01220053}).$\hfill\blacksquare$\par
	The next result is a uniform  moment estimate for the approximating equations $\{X^n\}_{n\geq 1}$.
	\begin{lemma}\label{estimate}
		For any $p\geq 2$, there exists a constant $C_p>0$ such that for any $n\geq 1$,
		\begin{align}
			E\Big[\sup_{t\in[0,T]}|X^n_t|_t^p\Big]+E\Big[\big(\int_{0}^{T}\|X^n_t\|_V^\alpha dt\big)^\frac{p}{2}\Big]\leq C_p\big(1+|\xi|_H^p\big).
		\end{align}
	\end{lemma}
	\noindent {\bf Proof}.
	Apply the $\mathrm{It\hat{o}}$ formula established in Theorem \ref{ito}, we then obtain:
	\begin{align}\nonumber
		&|X^n_t|_t^2\\=\ &|\xi^n|^2+2\int_{0}^{t}\big\langle P_n(s)A(s,X^n_s),\iota_s^*X^n_s\big\rangle ds+2\int_{0}^{t}\big(P_n(s)B(s,X^n_s)Q_ndW_s,X^n_s\big)_s\nonumber\\\ &+\int_{0}^{t}\|P_n(s)B(s,X^n_s)Q_n\|^2_{L_2(U,H^s)} ds+\int_{0}^{t}(X^n_s,\Phi(s)X_s^n)ds.\label{01232045}
	\end{align}
	Since $t\rightarrow |X^n_t|_t$ is continuous for $a.e.\ \omega$ due to our assumptions on the time-dependent norm $|\cdot|_t$, we can define $$\tau^n_M\dot{=}\inf\{t>0:|X^n_t|_t\geq M\ \text{or}\ \int_{0}^{t}\|X^n_s\|_V^\alpha ds\geq M\}. $$ 
	\ \ \ \ Now, we can apply the $\mathrm{It\hat{o}}$ formula to $f(|X^n_t|_t^2)$ with $f(x)=x^\frac{p}{2}$ for $p>2$, we then obtain, using the orthogonality of $P_n(s)$ in $H^s$,
	\begin{align}
		\nonumber&|X^n_t|_t^p\\\leq\ \nonumber&|\xi^n|^p+\frac{p}{2}\int_{0}^{t}|X^n_s|_s^{p-2}\big(2\langle A(s,X^n_s), \iota_s^*X_s^n \rangle+\|B(s,X^n_s)\|^2_{L_2(U,H^s)}\big)ds\\ \nonumber&+\frac{p}{2}\int_{0}^{t}|X_s^n|_s^{p-2}(X^n_s,\Phi(s)X^n_s)ds+p\int_{0}^{t}|X^n_s|_s^{p-2}\big(B(s,X^n_s)Q_ndW_s, X^n_s\big)_s\\ &+\frac{p(p-2)}{2}\int_{0}^{t}|X^n_s|_s^{p-4}\|\big(P_n(s)B(s,X^n_s)Q_n\big)^*X^n_s\|^2_{U}ds.\label{01250109}
	\end{align}
	By (\hyperlink{C1}{C1}),  (\hyperlink{H3}{H3}) and (\hyperlink{H5}{H5}), it follows that
	\begin{align}\label{eq 20240720}
		&\ |X^n_t|_t^p+\frac{pc}{2}\int_{0}^{t}|X^n_s|_s^{p-2}\|X^n_s\|_V^\alpha ds\\\nonumber\leq& \ |\xi^n|^p+\frac{p}{2}\int_{0}^{t}f(s)\big(1+|X_s^n|_s^2\big)|X^n_s|_s^{p-2}ds+p\int_{0}^{t}|X^n_s|_s^{p-2}\Big(B(s,X^n_s)Q_n dW_s, X^n_s\Big)_s\\\nonumber&\ +C_p\int_{0}^{t}\|\Phi(s)\|_{\mathcal{L}(H)}|X_s^n|_s^pds+C_p\int_{0}^{t}|X^n_s|_s^{p-2}\|B(s,X^n_s)\|_{L_2(U,H^s)}^2ds
		\\ \nonumber\leq&\ |\xi^n|^p+C_p\int_{0}^{t}(f(s)+\|\Phi(s)\|_{\mathcal{L}(H)})(1+|X^n_s|^p_s)ds
		+p\int_{0}^{t}|X^n_s|_s^{p-2}\Big(B(s,X^n_s)Q_n dW_s, X^n_s\Big)_s.
	\end{align}
	Replacing $t$ by $t\wedge\tau^n_M$ and taking expectation at both sides, we obtain by the BDG inequality \begin{align*}
		&E\big[\sup_{s\in[0,t\wedge\tau^n_M]}|X^n_{s}|_{s}^p\big]+\frac{pc}{2}E
		\big[\int_{0}^{t\wedge\tau^n_M}|X^n_s|_s^{p-2}\|X^n_s\|_V^\alpha ds\big]\\\leq\ & |\xi^n|^p+C_p\int_{0}^{t}\big(f(s)+\|\Phi(s)\|_{\mathcal{L}(H)}\big) (1+E\sup_{r\in[0,s\wedge\tau^n_M]}|X^n_{r}|_{r}^p )ds\\&\ +C_pE\Big(\int_{0}^{t\wedge\tau^n_M}|X^n_s|_s^{2p-2}\|B(s,X^n_s)\|^2_{L_2(U,H^s)}ds \Big)^{\frac{1}{2}}\\\leq\ & |\xi^n|^p+C_p\int_{0}^{t}\big(f(s)+\|\Phi(s)\|_{\mathcal{L}(H)}\big) \big(1+E\sup_{r\in[0,s\wedge\tau^n_M]}|X^n_{r}|^p_{r} \big)ds\\&\ +C_pE\Big[\sup_{s\in[0,t\wedge\tau^n_M]}|X^n_s|_s^{\frac{p}{2}}\Big(\int_{0}^{t\wedge\tau^n_M}f(s)|X^n_s|_s^{p-2}\big(1+|X^n_s|_s^2)ds \Big)^{\frac{1}{2}}\Big]\\\leq\ & |\xi^n|^p+C_p\int_{0}^{t\wedge\tau^n_M}\big(f(s)+\|\Phi(s)\|_{\mathcal{L}(H)}\big) \big(1+E\big[\sup_{r\in[0,s\wedge\tau^n_M]}|X^n_{r}|^p_{r} \big]\big)ds \\&\ +C_p\int_{0}^{t}f(s)E\big[\sup_{r\in[0,s\wedge\tau^n_M]}|X^n_{r}|^p_{r}\big]ds+\frac{1}{2} E\Big[\sup_{s\in[0,t\wedge\tau^n_M]}|X^n_s|_s^{p}\Big].
	\end{align*}
	Collecting the terms, we obtain
	\begin{align*}
		&E\big[\sup_{s\in[0,t\wedge\tau^n_M]}|X^n_{s}|_{s}^p\big]+pcE
		\big[\int_{0}^{t\wedge\tau^n_M}|X^n_s|_s^{p-2}\|X^n_s\|_V^{\alpha}ds\big]\\\leq&\ 2|\xi^n|^p+C_p\int_{0}^{t}\big(f(s)+\|\Phi(s)\|_{\mathcal{L}(H)}\big) \Big(E\big[\sup_{r\in[0,s\wedge\tau^n_M]}|X^n_r|_r^{p}\big]  +1\Big)ds.
	\end{align*}
	Applying the Gronwall lemma, there exists a constant $C_p>0$ independent of $M$ and $n$ such that
	$$E\big[\sup_{s\in[0,T\wedge\tau^n_M]}|X^n_{s}|_{s}^p\big]+E
	\big[\int_{0}^{T\wedge\tau^n_M}|X^n_s|_s^{p-2}\|X^n_s\|_V^{\alpha}ds\big]\leq\ C_p. $$
	By letting $M\rightarrow\infty$, we then obtain
	\begin{align}\label{01250244}
		E\big[\sup_{s\in[0,T]}|X^n_{s}|_{s}^p\big]+E
		\big[\int_{0}^{T}|X^n_s|_s^{p-2}\|X^n_s\|_V^{\alpha}ds\big]\leq\ C_p.
	\end{align}
	Let $p=2$ in (\ref{eq 20240720}) to obtain
	\begin{align*}
		\nonumber&\ \  |X^n_t|_t^2+c\int_{0}^{t}\|X^n_s\|_V^{\alpha}ds\\\nonumber\leq&\ |\xi^n|^2+C\int_{0}^{t}(f(s)+\|\Phi(s)\|_{\mathcal{L}(H)})(1+|X^n_s|^2_s)ds
		+p\int_{0}^{t}\Big(B(s,X^n_s)Q_n dW_s, X^n_s\Big)_s.
	\end{align*}
	This implies that
	\begin{align*}
		&E\Big[\big(\int_{0}^{T}\|X^n_s\|_V^\alpha ds\big)^\frac{p}{2}\Big]\\\leq&\  C_p|\xi^n|^p+C_p\big[\int_{0}^{T}\big(f(s)+\|\Phi(s)\|_{\mathcal{L}(H)}\big)ds\big]^\frac{p}{2}\big(1+E\sup_{t\in[0,T]}|X^n_t|^p_t\big)\\&\ +\ C_pE\Big[\Big(\int_{0}^{T}\|Q_n^*B(t,X^n_t)^*\iota_t^*X^n_t\|^2_{U}dt\Big)^\frac{p}{4}\Big]\\\leq&\ C_p|\xi^n|^p+C_p\big[\int_{0}^{T}(f(s)+\|\Phi(s)\|_{\mathcal{L}(H)})ds\big]^\frac{p}{2}\big(1+E\sup_{t\in[0,T]}|X^n_t|^p_t\big)\\&\ +\ C_pE\Big[\sup_{t\in[0,T]}|X^n_t|_t^p+1\Big]\Big(\int_{0}^{T}f(t)dt\Big)^\frac{p}{2}.
	\end{align*}
	Combining with (\ref{01250244}), we complete the proof of the lemma.
	$\hfill\blacksquare$\\
	Denote
	\begin{align*}
		K\dot{=}\ L^\alpha\big([0,T]\times\Omega, dt\otimes d\mathbb{P}, V\big),\ J\dot{=}\ L^2\big([0,T]\times\Omega, dt\otimes d\mathbb{P}, L_2(U,H)\big).
	\end{align*}
	\ \  \ \ \ As a corollary of the above lemma and the reflexivity of the Banach space $V$, there exists a subsequence of $\{X^n\}_{n\geq 1}$ (still denoted by $\{X^n\}_{n\geq 1}$ for simplicity) and
	\begin{align}\label{02192216}
		X\in L^p\big(\Omega,L^\infty([0,T],H)\big)\cap L^\frac{\alpha p}{2}\big(\Omega, L^\alpha([0,T],V)\big)\cap K ,\ \forall p\geq 2,
	\end{align} such that
	\begin{align}
		&X^{n}\rightarrow {X}\text{ weakly in $K$},\label{012323121}\\&
		X^{n}\rightarrow {X}\text{ weakly in $L^2([0,T]\times\Omega, dt\otimes d\mathbb{P}, H )$},\label{012323122}\\&
		X^{n}\rightarrow {X}\text{ in the weak $*$ topology of  $L^2\big(\Omega, L^\infty([0,T], H)\big)$}.\label{012323123}
	\end{align}
	
	By Lemma \ref{estimate} again, and  due to (\ref{eq 20240720 00}), (\hyperlink{H4}{H4}) and (\hyperlink{H5}{H5}), there exist a subsequence (still labeled by $n\in \mathbb{N}$), $Y\in K^*=L^\frac{\alpha}{\alpha-1}\big([0,T]\times\Omega,\ dt\otimes d\mathbb{P},\ V^*\big)$ and $Z\in J$ such that
	\begin{align}
		\label{012400291}&A(\cdot,X^n_\cdot)\rightarrow Y\ \text{weakly in $K^*$},\\
		\label{012400292}&P_{n}(\cdot)B(\cdot,X^n_\cdot)Q_n\rightarrow Z\ \text{weakly in $J$},\\
		\label{012400293}&\int_{0}^{\cdot}P_n(s)B(s,X^n_s)Q_ndW_s\rightarrow \int_{0}^{\cdot}Z(s)dW_s, \text{weakly in $\mathcal{M}_T^2(H)$}\nonumber\\
		&\ \ \ \ \ \ \ \ \ \ \ \ \ \ \ \ \ \ \text{and in the weak $*$ topology of $L^\infty\big([0,T],L^2(\Omega,H)\big)$},
	\end{align}
	where $\mathcal{M}_T^2(H)$ is the space of all $H$-valued continuous, square integrable martingales $M_\cdot$ equipped with the norm
	$$\|M\|_T\dot{=}\Big(E|M_T|^2\Big)^\frac{1}{2}=\sup_{t\in[0,T]}\Big(E|M_t|^2\Big)^\frac{1}{2}.$$
	Note, since
	$$dX^n_t=P_n(t)A(t,X^n_t)dt+P_n(t)B(t,X^n_t)Q_ndW_t,$$
	as the proof of (\ref{01202028}) we can show that
	$$d\iota_t^*X^n_t=\iota^*_tP_n(t)A(t,X^n_t)dt+\iota_t^*P_n(t)B(t,X^n_t)Q_ndW_t+\Phi(t)X^n_tdt.$$
	Then for any
	\begin{align*}
		v\in\mathop{\scalebox{1.5}[1.5]{{$\cup$}}}_{n=1}^\infty H_n,\phi\in L^\infty([0,T]\times\Omega),\ \text{we have}
	\end{align*}
	\begin{align}\label{01232315}
		&E\Big[\int_{0}^{T} \big(\iota_t^*X^n_t, v\big) \phi(t)dt\Big]\\=&\ E\Big[\int_{0}^{T}\phi(t)dt \big( \xi^n , v\big)\Big]+E\Big[\int_{0}^{T}\int_{0}^{t}\big( \iota_s^*P_n(s)A(s,X^n_s),v\big)ds \phi(t)dt\Big]\nonumber\\&+E\Big[\int_{0}^{T}\int_{0}^{t}\Big( \iota_s^*P_n(s)B(s,X^n_s)Q_ndW_s,v\Big) \phi(t)dt\Big]\nonumber\\&+E\Big[\int_{0}^{T}\int_{0}^{t}\big( \Phi(s)X^n_s, v\big) ds\phi(t)dt\Big]\nonumber\\=&\  \mathrm{I}_n+\cdots+\mathrm{IV}_n\nonumber.
	\end{align}
	Letting $n\rightarrow \infty$, in view of  (\ref{012323121}), the left hand side of (\ref{01232315}) tends to
	\begin{align*}
		E\Big[\int_{0}^{T} \big(\iota_t^*X_t, v\big) \phi(t)dt\Big].
	\end{align*}
	Denote $\eta(t)=\int_{t}^{T}\phi_sds,\ t\in[0,T]$, then as $n\rightarrow\infty$,
	\begin{align*}
		&\mathrm{I}_n\rightarrow E\big[\eta(0)(\xi,v)\big]= E\Big[\int_{0}^{T}\phi(t)dt \big( \xi , v\big)\Big].
	\end{align*}
	By (\ref{01212124}) and (\ref{012400291}), as $n\rightarrow\infty$,
	\begin{align*}
		\mathrm{II}_n= E\Big[\int_{0}^{T}\eta(s)\big( \iota_s^*P_n(s)A(s,X^n_s),v\big)ds\Big]\rightarrow E\Big[\int_{0}^{T}\eta(s)\big( \iota_s^*Y_s,v\big)ds\Big].
	\end{align*}
	By (\ref{012400292}) and (\hyperlink{C1}{C1}), we have
	\begin{align}\label{01240056}
		\iota^*_\cdot P_{n}(\cdot)B(\cdot,X^n_\cdot)Q_n\rightarrow \iota_\cdot^*Z_\cdot\ \text{weakly in $J$},
	\end{align}
	which implies
	\begin{align}
		&\int_{0}^{\cdot}\iota_s^*P_n(s)B(s,X^n_s)Q_ndW_s\rightarrow \int_{0}^{\cdot}\iota_s^*Z_sdW_s,
	\end{align}
	in the weak $*$ topology of $L^\infty\big([0,T],L^2(\Omega,H)\big)$. Thus as $n\rightarrow\infty$,
	\begin{align*}
		\mathrm{III}_n=\ &E\Big[\int_{0}^{T}\Big(\int_{0}^{t} \iota_s^*P_n(s)B(s,X^n_s)Q_ndW_s,\phi(t)v\Big) dt\Big]
		\\\rightarrow\ &E\Big[\int_{0}^{T}\Big(\int_{0}^{t} \iota_s^*Z_sdW_s,\phi(t)v\Big) dt\Big].
	\end{align*}
	By (\ref{012323121}), as $n\rightarrow \infty$
	\begin{align*}
		\mathrm{IV}_n=\ &E\Big[\int_{0}^{T}\Big(\Phi(s)X^n_s ,\eta(s)v\Big) ds\Big]
		\rightarrow E\Big[\int_{0}^{T}\Big(\Phi(s)X_s ,\eta(s)v\Big) ds\Big].
	\end{align*}
	Putting the above limits together, we obtain
	\begin{align*}
		E\Big[\int_{0}^{T} \big(\iota_t^*X_t, \phi(t)v\big) dt\Big]=\ E\Big[\int_{0}^{T}\langle\xi+\int_{0}^{t}\iota_s^*Y_s ds+\int_{0}^{t}\iota_s^*Z_sdW_s+\int_{0}^{t}\Phi(s)X_s,\phi(t)v\rangle dt\Big].
	\end{align*}
	Due to the arbitrariness of $\phi$ and $v$, we find the following equality in $V^*$,
	\begin{align}\label{05110042}
		\iota_t^*X_t=\xi+\int_{0}^{t}\iota_s^*Y_sds+\int_{0}^{t}\iota_s^*Z_s dW_s+\int_{0}^{t}\Phi(s)X_sds.
	\end{align}
	Because $\iota_\cdot^*X_\cdot\in L^\alpha([0,T]\times\Omega,V)\cap L^2\big(\Omega,L^\infty([0,T],H)\big)$ and our assumptions on $\iota_\cdot^*$ and $\Phi(\cdot)$ $\big($see (\hyperlink{C2}{C2}) and (\hyperlink{C3}{C3})$\big)$, we assert that $\iota_\cdot^*X_\cdot$ has a continuous trajectory in $H$, see e.g., Theorem 3.1 in \cite{P}. By Lemma \ref{app} in the Appendix A, we then have  for $\forall\ e\in V$, $t\in[0,T]$,
	\begin{align*}
		&\big( X_t, e \big)=\langle \iota_t^*X_t,\iota_{-t}^*e\rangle\\=\ &\langle \xi, e\rangle +\int_{0}^{t}\langle \iota_{s}^*Y_s, \iota_{-s}^*e\rangle ds+\int_{0}^{t}\big( \iota_s^*Z_sdW_s,\iota_{-s}^*e\big)+\int_{0}^{t}\big(\Phi(s)X_s,\iota_{-s}^*e\big)ds\\&\ -\int_{0}^{t}(\iota_s^*X_t,\iota_{-s}^*\Phi(s)\iota_{-s}^*e)dt\\=\ &\langle\xi+\int_{0}^{t} Y_s  ds+\int_{0}^{t}Z_sdW_s,e\rangle.
	\end{align*}
	Thus we obtain for any $t\in[0,T]$,
	\begin{align}\label{01280342}
		X_t=\xi+\int_{0}^{t}Y_sds+\int_{0}^{t}Z_sdW_s.
	\end{align}
	Then, to complete the proof of the existence of the solution of equation(\ref{spde}), it remains to show that
	\begin{align}\label{01242004}
		Y_\cdot={A}(\cdot,X_\cdot)\text{ and }Z_\cdot={B}(\cdot,X_\cdot).
	\end{align}
	The following two lemmas are needed.
	
	\begin{lemma}\label{01250332}
		Let $Z_\cdot\in J$ be given as in (\ref{012400292}). We have
		\begin{align*}
			\iota_\cdot^*B\big(\cdot,X^n_\cdot\big)\rightarrow \iota_\cdot^*Z_\cdot\text{ weakly in $J$ as $n\rightarrow\infty$.}
		\end{align*}
	\end{lemma}
	\noindent {\bf Proof}.
	Indeed, for any $\mathcal{B}\in J$,
	\begin{align*}
		&E\int_{0}^{T}\big(\iota^*_sB(s,X^n_s), \mathcal{B}_s\big)_{L_2(U,H)} ds\\=\ & E\int_{0}^{T}\Big(\iota^*_s(I-P_n(s))B(s,X^n_s), \mathcal{B}_s\Big)_{L_2(U,H)} ds+E\int_{0}^{T}\Big(\iota^*_sP_n(s)B(s,X^n_s)Q_n, \mathcal{B}_s\Big)_{L_2(U,H)} ds\\ & +E\int_{0}^{T}\Big(\iota^*_sP_n(s)B(s,X^n_s)\big(I-Q_n\big), \mathcal{B}_s\Big)_{L_2(U,H)} ds\\=\ &\mathrm{I}_n+\mathrm{II}_n+\mathrm{III}_n.
	\end{align*}
	As $n\rightarrow\infty$, by (\ref{01240056}),
	\begin{align*}
		&\mathrm{II}_n\rightarrow E\int_{0}^{T}\Big(\iota^*_sZ_s, \mathcal{B}_s\Big)_{L_2(U,H)} ds.
	\end{align*}
	Noticing that  $\|\cdot\|_{L_2(U,H^s)}\leq C\|\cdot\|_{L_2(U,H)} $ for some $C>0$, the self-adjointness of $P_n(s)$ in $H^s$ and that $P_n(s)u\rightarrow u$ as $n\rightarrow\infty$ in $|\cdot|_s$ for any $s\in[0,T]$ and $u\in H$, we easily deduce that, as $n\rightarrow\infty$,
	\begin{align*}
		&\mathrm{I}_n= E\int_{0}^{T}\Big(B(s,X^n_s), \big(I-P_n(s)\big)\mathcal{B}_s\Big)_{L_2(U,H^s)} ds\rightarrow 0,\\&
		\mathrm{III}_n=E\int_{0}^{T}\Big(B(s,X^n_s), P_n(s)\mathcal{B}_s\big(I-Q_n\big)\Big)_{L_2(U,H^s)} ds\rightarrow 0,
	\end{align*}
	completing the proof.$\hfill\blacksquare$\par
	
	\begin{lemma}
		For any $\psi\in L^\infty([0,T],\mathbb{R}_+)$, we have
		\begin{align}\label{01242112}
			E\big[\int_{0}^{T}\psi(t)|X_t|_t^2dt\big]\leq \liminf_{n\rightarrow\infty} E\big[\int_{0}^{T}\psi(t)|X^{n}_t|_t^2dt\big].
		\end{align}
	\end{lemma}
	\noindent {\bf Proof}.
	By $(\ref{012323122})$,
	\begin{align*}
		&E\big[\int_{0}^{T}\psi(t)|X_t|_t^2dt\big]= E\big[\int_{0}^{T}\psi(t)\big(\iota_t^*X_t,X_t\big)dt\big]\\=&\ \lim\limits_{n\rightarrow\infty} E\big[\int_{0}^{T}\psi(t)\big(\iota_t^*{X}^n_t,X_t\big)dt\big]=\lim\limits_{n\rightarrow\infty} E\big[\int_{0}^{T}\psi(t)\big({X}^n_t,X_t\big)_tdt\big]\\\leq&\  E\big[\int_{0}^{T}\psi(t)|X_t|_t^2dt\big]^\frac{1}{2}\liminf\limits_{n\rightarrow\infty}E\big[\int_{0}^{T}\psi(t)|{X}^n_t|_t^2dt\big]^\frac{1}{2},
	\end{align*}
	which implies (\ref{01242112}).$\hfill\blacksquare$\par
	Now, we are ready to prove the claim (\ref{01242004}). Take any $H$-valued progressively measurable process $\phi_\cdot\in K\cap L^2\big(\Omega, L^\infty([0,T],H)\big)$, we can then define
	\begin{align}\label{01242351}
		\Psi(t)=\int_{0}^{t}\big( f(s)+c_1^2\|\Phi(s)\|_{\mathcal{L}(H)}+\rho(\phi_s)\big)ds.
	\end{align}
	Applying the $\mathrm{It\hat{o}}$ formula to $e^{-\Psi(t)}|X_t|_t^2 $ and $e^{-\Psi(t)}|X^n_t|_t^2 $ respectively, and then taking  expectation, we obtain,
	\begin{align}\label{01250338}
		&E\Big[ e^{-\Psi(t)}|X_t|_t^2\Big]-|\xi|^2\\=\ &E\Big[\int_{0}^{t}e^{-\Psi(s)}\big(2\langle Y_s,\iota_s^*X_s\rangle+\|Z_s\|_{L_2(U,H^s)}^2+(\Phi(s)X_s,X_s)-\Psi'(s)|X_s|_s^2 \big) ds\Big]\nonumber.
	\end{align}
	and
	\begin{align*}
		&\ E\Big[e^{-\Psi(t)}|X^n_t|_t^2\Big]-|\xi^n|^2\\\leq&\  E\Big[\int_{0}^{t}e^{-\Psi(s)}\Big(2\langle A(s,X^{n}_s),\iota_s^*X_s^n\rangle+\|B(s,X_s^n)\|^2_{L_2(U,H^s)}+\big(\Phi(s)X^n_s,X^n_s\big)\Big)ds\Big]\\& \ -E\Big[\int_{0}^{t}e^{-\Psi(s)}\Psi'(s)|X^n_s|_s^2ds\Big]\\=&\ E\Big[\int_{0}^{t}e^{-\Psi(s)}\Big(2\langle A(s,X^{n}_s)-A(s,\phi_s),\iota_s^*X_s^n-\iota_s^*\phi_s\rangle+\|B(s,X_s^n)-B(s,\phi_s)\|^2_{L_2(U,H^s)}\\&\ +\big(\Phi(s)(X^n_s-\phi_s),X^n_s-\phi_s\big)\Big)ds\Big] -E\Big[\int_{0}^{t}e^{-\Psi(s)}\Psi'(s)|X^n_s-\phi_s|_s^2ds\Big]\\&\ +E\Big[\int_{0}^{t}e^{-\Psi(s)}\Big(2\langle A(s,\phi_s),\iota_s^*X^n_s\rangle+2\langle A(s,X^n_s)-A(s,\phi_s),\iota_s^*\phi_s\rangle \\&\ -\|B(s,\phi_s)\|_{L_2(U,H^s)}^2+2\langle B(s,X^n_s),B(s,\phi_s)\rangle_{L_2(U,H^s)}-2\Psi'(s)\big(X_s^n,\phi_s\big)_s\\&\ +\Psi'(s)|\phi_s|_s^2+2\big(\Phi(s)X_s^n,\phi_s\big)-\big(\Phi(s)\phi_s,\phi_s\big)\Big)ds\Big].
	\end{align*}
	By (\hyperlink{C1}{C1}), (\hyperlink{H2}{H2}) and $(\ref{01242351})$, the right hand side is less than
	\begin{align*}
		&\ E\Big[\int_{0}^{t}e^{-\Psi(s)}\Big(2\langle A(s,\phi_s),\iota_s^*X^n_s\rangle+2\langle A(s,X^n_s)-A(s,\phi_s),\iota_s^*\phi_s\rangle \\&\ \ \ \   -\|B(s,\phi_s)\|_{L_2(U,H^s)}^2+2\langle B(s,X^n_s),B(s,\phi_s)\rangle_{L_2(U,H^s)}-2\Psi'(s)\big(X_s^n,\phi_s\big)_s\\&\ \ \ \  +\Psi'(s)|\phi_s|_s^2+2\big(\Phi(s)X_s^n,\phi_s\big)-\big(\Phi(s)\phi_s,\phi_s\big)\Big)ds\Big].
	\end{align*}
	Taking $\psi\in L^\infty([0,T],\mathbb{R}_+)$, from the above estimate,  we  obtain
	\begin{align*}
		&\ E\Big[\int_{0}^{T}\psi(t)\big(e^{-\Psi(t)}|X^n_t|_t^2-|\xi|^2\big)dt\Big]\\\leq&\ E\Big[\int_{0}^{T}\psi(t)\int_{0}^{t}e^{-\Psi(s)}\Big(2\langle A(s,\phi_s),\iota_s^*X^n_s\rangle+2\langle A(s,X^n_s)-A(s,\phi_s),\iota_s^*\phi_s\rangle \\&\ \ \ \   -\|B(s,\phi_s)\|_{L_2(U,H^s)}^2+2\langle B(s,X^n_s),B(s,\phi_s)\rangle_{L_2(U,H^s)}-2\Psi'(s)\big(X_s^n,\phi_s\big)_s\\&\ \ \ \  +\Psi'(s)|\phi_s|_s^2+2\big(\Phi(s)X_s^n,\phi_s\big)-\big(\Phi(s)\phi_s,\phi_s\big)\Big)dsdt\Big].
	\end{align*}
	Combining (\ref{012323121})-(\ref{012323123}) and   (\ref{012400291}) with Lemma \ref{estimate} and Lemma \ref{01250332}, we deduce that
	\begin{align}\label{05120250}
		&\ E\Big[\int_{0}^{T}\psi(t)\big(e^{-\Psi(t)}|X_t|_t^2-|\xi|^2\big)dt\Big]\\\leq
		&\nonumber \liminf_{n\rightarrow\infty}E\Big[\int_{0}^{T}\psi(t)\big(e^{-\Psi(t)}|X^n_t|_t^2-|\xi|^2\big)dt\Big]\\\leq&\ \nonumber E\Big[\int_{0}^{T}\psi(t)\int_{0}^{t}e^{-\Psi(s)}\Big(2\langle A(s,\phi_s),\iota_s^*X_s\rangle+2\langle A(s,X_s)-A(s,\phi_s),\iota_s^*\phi_s\rangle \\&\ \ \ \nonumber\   -\|B(s,\phi_s)\|_{L_2(U,H^s)}^2+2\langle B(s,X_s),B(s,\phi_s)\rangle_{L_2(U,H^s)}-2\Psi'(s)\big(X_s,\phi_s\big)_s\\&\ \ \ \  +\Psi'(s)|\phi_s|_s^2+2\big(\Phi(s)X_s,\phi_s\big)-\big(\Phi(s)\phi_s,\phi_s\big)\Big)dsdt\Big].
		{}\end{align}
	Combining (\ref{05120250}) with (\ref{01250338}), we obtain that  for $\forall\ \phi_\cdot\in K\ \cap\  L^2(\Omega,L^\infty\big([0,\infty],H)\big)$,
	\begin{align}
		\nonumber&E\Big[\int_{0}^{T}\psi(t)\int_{0}^{t}e^{-\Psi(s)}\Big(2\langle Y_s-A(s,\phi_s),\iota_s^*({X}_s-\phi_s)\rangle+\|B(s,\phi_s)-Z_s\|^2_{L_2(U,H^s)}\\&\ +\big(\Phi(s)(X_s-\phi_s),X_s-\phi_s\big)-\Psi'(s)|X_s-\phi_s|_s^2 \Big)dsdt\Big]\leq 0.\label{01250426}
	\end{align}
	We first let $\phi_\cdot=X_\cdot$, to obtain
	\begin{align*}
		Z_\cdot=B(\cdot,X_\cdot). \text{\ \ \ $dt\otimes d\mathbb{P}$ a.s.}
	\end{align*}
	Taking $\phi=X-\varepsilon\tilde{\phi}v$ in (\ref{01250426}), for some $\tilde{\phi}\in L^\infty([0,T],\mathbb{R})$ and $v\in V$, $\varepsilon>0$, we have
	\begin{align}
		\nonumber&E\Big[\int_{0}^{T}\psi(t)\int_{0}^{t}e^{-\Psi(s)}\tilde{\phi}_s\Big(2\varepsilon\langle Y_s-A(s,X_s-\varepsilon\tilde{\phi}_sv),\iota_s^*v\rangle\\&\ +\varepsilon^2\tilde{\phi}_s^2\big(\Phi(s)v,v\big)-\varepsilon^2\Psi'(s)\tilde{\phi}_s^2|v|_s^2 \Big)dsdt\Big]\leq 0.\label{01251722}
	\end{align}
	Dividing both sides by $\varepsilon$, we have
	\begin{align}
		\nonumber&E\Big[\int_{0}^{T}\psi(t)\int_{0}^{t}e^{-\Psi(s)}\tilde{\phi}_s\Big(2\langle Y_s-A(s,X_s-\varepsilon\tilde{\phi}_sv),\iota_s^*v\rangle\\&\ +\varepsilon\tilde{\phi}_s^2\big(\Phi(s)v,v\big)-\varepsilon\Psi'(s)\tilde{\phi}_s^2|v|_s^2 \Big)dsdt\Big]\leq 0.\label{01251957}
	\end{align}
	Letting $\varepsilon\rightarrow0$, by (\hyperlink{H1}{H1}), (\hyperlink{H4}{H4}) and (\ref{02192216}), we get
	\begin{align}
		\nonumber&E\Big[\int_{0}^{T}\psi(t)\int_{0}^{t}e^{-\Psi(s)}\tilde{\phi}_s\langle Y_s-A(s,X_s),\iota_s^*v\rangle dsdt\Big]\leq 0.
	\end{align}
	The arbitrariness of $\tilde{\phi}$, $v$ and $\psi$ implies that
	$$\iota_\cdot^* Y_\cdot=\iota_\cdot^*A(\cdot,X_\cdot),$$
	which yields $Y_\cdot=A(\cdot,X_\cdot)$ taking into account (\hyperlink{C4}{C4}). The proof is complete.$\hfill\blacksquare$
	\section{Stochastic Stefan type problem on a moving hypersurface}\par
	\ \ \ \
	\setcounter{equation}{0}
	In this section, we will present the application of the result established in Section 2 to solve the stochastic Stefan type problem on  moving hypersurfaces. The transformed version of this problem  on a fixed domain will fit into the framework provided in Section 2, but not the usual monotonicity framework as in \cite{LR2} or \cite{RSZ}. A rigorous derivation of the transformed SPDE on the fixed hypersurface as well as the equivalence of the solutions are provided in Subsection 3.5 below. We would like to mention that the Stefan type problem we considered also includes the stochastic porous media equation with $p\in[\frac{2n}{n+2},\frac{2n}{n-2}]\big(\Psi(s)=|s|^{p-2}s\ \text{in}\ (\ref{pme})\big)$.\par
	Later on, as numerous constant estimates are involved, we will denote by $A_{t,u}\approx B_{t,u}$ if there exists a constant $C>0$ independent of $t\in[0,T] $ and element $u$ such that $ A_{t,u}\in[C^{-1}B_{t,u},C B_{t,u}]$. Analogously, we denote $A_{t,u}\lesssim B_{t,u}$ for the one-sided inequaility.\par
	\subsection{Preliminaries on moving hypersurfaces}
	We begin with some preliminaries on the hypersurfaces. We will sketch the framework outlined in Subsection 2.1 of \cite{DE}. Let $n\in \mathbb{N}$. $\Gamma\subseteq \mathbb{R}^{n+1}$ is an $n$-dimensional compact $C^k$ hypersurface if there exists a finite collection of local parametrizations $\big\{X^i\big \}_{i\leq N}$ such that for each $i\leq N$, $X^i\in C^k(U^i,\mathbb{R}^{n+1})$, where $U^i$ are connected open sets in $\mathbb{R}^n$, satisfying the following conditions:
	\begin{itemize}{}
		\item [(i)]: $X^i$ is injective and rank$(\nabla X^i(\theta))=n$ for any $\theta\in U^i$ and $i\leq N$.
		\item [(ii)]: For each $i$, there exists an open set $\tilde{V}^i$ in ${\mathbb{R}^{n+1}}$ such that $X^i(U^i)=\tilde{V}^i\cap \Gamma$.
		\item[(iii)]: $\{V^i\ \dot{=}\ X^i(U^i)\}_{i\leq N}$ is an open cover of $\Gamma$.
	\end{itemize}
	The definition of a compact $C^k$ hypersurface implies that $\Gamma$ is a compact $C^k$ submanifold in $\mathbb{R}^{n+1}$. Consequently, it inherits both the topology and the $C^k$ structure from $\mathbb{R}^{n+1}$, with local coordinates $\{(X^i)^{-1}\}_{i\leq N}:V^i\ \dot{=}\ X^i(U^i)\rightarrow U^i$. Also, it inherits the canonical Riemannian metric on $\mathbb{R}^{n+1}$ as a submanifold.\par
	Let $\alpha\leq N$, $\theta\in U^\alpha$. We consider the induced Riemannian metric on $U^\alpha$, given by
	$$g^\alpha_{ij}(\theta)=\frac{\partial X^\alpha_k}{\partial \theta_i}(\theta)\frac{\partial X^\alpha_k}{\partial \theta_j}(\theta),\ i,j= 1,\cdots, n. $$
	We also denote the inverse of $\{g^\alpha_{ij}(\theta)\}_{1\leq i,j\leq n}$ as $\{g_\alpha^{ij}(\theta)\}_{1\leq i,j\leq n}$ and  det$(g^\alpha_{ij}(\theta))$  as $g^\alpha(\theta)$ for $\theta\in U^\alpha$.\par
	A function $f:\Gamma\rightarrow\mathbb{R}$ is k-times continuously differentiable if all the functions $f\circ X^i:U^i\rightarrow\mathbb{R}$, $i\leq N$, are k-times continuously differentiable. We denote by $C^1(\Gamma)$ the set of functions $f:\Gamma\rightarrow\mathbb{R}$ that is continuously differentiable and similarly for $C^2(\Gamma)$. \par
	Using local charts, one can express the Laplace-Beltrami operator on $\Gamma$ locally. For $f\in C^2(\Gamma)$, we have
	$$\big(\Delta_\Gamma f\big)(X^\alpha(\theta))=\frac{1}{\sqrt{g^\alpha(\theta)}}\frac{\partial}{\partial \theta_j}\Big(g_\alpha^{ij}(\theta)\sqrt{g^\alpha(\theta)}\frac{\partial f}{\partial \theta_i}(\theta)\Big), $$
	where  $\frac{\partial f}{\partial \theta_i}$ denotes $\frac{\partial f\circ X^\alpha}{\partial \theta_i}(\theta)$ for $\theta\in U^\alpha$.
	The tangential gradient of $f:\Gamma\rightarrow\mathbb{R}$ is locally given by
	\begin{align}\label{05080028}
		(\nabla_\Gamma f)_k\big(X^\alpha(\theta)\big)=g_\alpha^{ij}(\theta)\frac{\partial f}{\partial \theta_j}\frac{\partial X^\alpha_k}{\partial \theta_i}(\theta),\  k=1,\cdots,n+1.
	\end{align}
	From (\ref{05080028}), we can deduce that for any $\alpha\leq N$, $\theta\in U^\alpha$,
	\begin{align}\label{08010005}
		\frac{\partial f}{\partial\theta_i}(\theta)=(\nabla_\Gamma f)_k(X^\alpha(\theta))\frac{\partial X_k^\alpha}{\partial \theta_i}(\theta).
	\end{align}
	Remember the convention that the appearance of  repeated index means summation over the index.
	As functions on $\Gamma$, we can show that $\Delta_\Gamma f$ and $\nabla_\Gamma f$ are independent of the choice of local charts. Sometimes, we will write $(\nabla_\Gamma f)_k$ as $\underline{D}_k f$ for short.
	Then, the divergence of $g:\Gamma\rightarrow\mathbb{R}^{n+1}$ is given by
	$$\nabla_\Gamma\cdot g=\sum_{i=1}^{n+1}\underline{D}_i g^i.$$
	Moreover, it holds that $$\Delta_\Gamma f=\sum_{i=1}^{n+1}\underline{D}_i\underline{D}_if=\nabla_\Gamma\cdot\nabla_\Gamma f.$$
	\par
	For $p\in[1,\infty)$, we let $L^p(\Gamma)$  denote the space of measurable functions $f:\Gamma\rightarrow\mathbb{R}$, that have finite $\|f\|_{L^p(\Gamma)}$ norm defined by
	$$\|f\|_{L^p(\Gamma)}=\Big(\int_{\Gamma}|f|^p(y)\mathrm{dvol}_\Gamma(dy)\Big)^\frac{1}{p}.$$
	In the above formula, $\mathrm{dvol}_\Gamma(dy)$ denotes the volume measure corresponding to the Riemannian metric $g_{ij}$ on $\Gamma$. By means of local charts, it can be written as
	$$\|f\|_{L^p(\Gamma)}^p=\sum_{\alpha=1}^{N}\int_{U^\alpha}\chi^\alpha(X^\alpha(\theta))|f|^p(X^\alpha(\theta))\sqrt{g^\alpha(\theta)}d\theta,$$
	where $\{\chi^\alpha\}_{\alpha\leq N}$ is a partition of unity with regard to the open cover $\{V^\alpha\}_{\alpha\leq N}$ on $\Gamma$.
	$L^p(\Gamma)$ is a Banach space for $p\geq1$ and for $p=2$ a Hilbert space. For $1\leq p<\infty$, the space $C^1(\Gamma)$ is dense in $L^p(\Gamma)$.
	The Sobolev space $H^{1}(\Gamma)$ is the closure of $C^1(\Gamma)$ under the norm
	$$ \|f\|_{H^1(\Gamma)}\dot{=}\big(\|f\|_{L^2(\Gamma)}^2+\|\nabla_\Gamma f\|^2_{L^2(\Gamma)}\big)^\frac{1}{2},$$
	and the space $H^{2,p}(\Gamma)$ is the closure of $C^2(\Gamma)$ under the norm
	\begin{align}\label{05081740}
		\|f\|_{H^{2,p}(\Gamma)}\dot{=}\big(\|f\|_{L^p(\Gamma)}^p+\|\nabla_\Gamma f\|^p_{L^p(\Gamma)}+\sum_{i,j=1}^{n+1}\|\underline{D}_i\underline{D}_jf\|^p_{L^p(\Gamma)}\big)^\frac{1}{p}.
	\end{align}
	The dual of $H^{1}(\Gamma)$ will be denoted as $H^{-1}(\Gamma)$,
	equipped with the norm
	\begin{align*}
		\|f\|_{{H}^{-1}(\Gamma)}=\sup_{v\in H^1(\Gamma),v\neq 0}\frac{ _{{H}^{-1}(\Gamma)}\langle f,v\rangle_{{H}^{1}(\Gamma)}}{\centering{\|v\|}_{H^1(\Gamma)}}.
	\end{align*}

	\subsection{The geometry of moving hypersurfaces}
	In this subsection, we will provide a brief overview of the moving hypersurfaces as outlined in Section 5 of \cite{DE}. Let $\{\Gamma_t\}_{t\in[0,T]}$ be a family of $C^3$ compact hypersurfaces indexed by $t\in[0,T]$. We assume that there exists a map $G(\cdot,\cdot):[0,T]\times \Gamma_0\rightarrow \mathbb{R}^{n+1}$, where $G\in C^1\big([0,T], C^2(\Gamma_0)\big)$, satisfying  that $G(t,\cdot):\Gamma_0\rightarrow\Gamma_t$ is a $C^2$ diffeomorphism between $\Gamma_0$ and $\Gamma_t$ for every $t\in[0,T]$ and $G(0,\cdot)=I_d$. Furthermore, we define the velocity of the evolving of $\Gamma_t$ by
	\begin{align}\label{08270324}
		v\big(t,G(t,\cdot)\big)=\frac{\partial G}{\partial t}(\cdot,t).
	\end{align}
	In accordance with Subsection 3.1, if $\{X^{i}(\theta), i\leq N\}$ is a local parametrization of $\Gamma_0$, then   $\Gamma_t$ has a family of local parametrization given by $\{X^{t,i}:U^i\rightarrow G(t,V^i)\}_{i\leq N}$ where $V^i=X^i(U^i)$, $X^{t,i}(\theta)=G(t,X^i(\theta))$ for $\theta\in U^i$ and the induced Riemannian metric
	\begin{align}\label{05082113}
		g_{ij}^{t,\alpha}(\theta)=\frac{\partial X_k^{t,\alpha}}{\partial \theta_i}(\theta)\frac{\partial X_k^{t,\alpha}}{\partial \theta_j}(\theta), \text{ $\theta$ $\in$ $U^\alpha$, $\alpha\leq N$.}
	\end{align}
	Moreover, its determinant det$\big(g_{ij}^{t,\alpha}(\theta)\big)$  is denoted by $g^\alpha(t,\theta)$ and its inverse by $g^{ij}_{t,\alpha}(\theta)$ for $\theta\in U^\alpha$.\par
	Consider the space-time surface given by
	\begin{align}\label{05120430}
		Q_T=\ \mathop{\scalebox{1.5}[1.5]{$\cup$}}_{t\in[0,T]} \{t\}\times\Gamma_t,
	\end{align}
	which is considered as an immersed submanifold of $[0,T]\times\Gamma_0$, induced by the map
	$$\tilde{G}:[0,T]\times\Gamma_0\rightarrow \mathbb{R}_+\times\mathbb{R}^{n+1},\ \tilde{G}(t,y)=\big(t,G(t,y)\big). $$
	An appropriate time derivative, which is usually called ``material derivative" and is denoted as $\partial^\bullet$, is defined on this moving surface $Q_T$, that is,
	\begin{align}\label{md}
		\partial^\bullet f\dot{=}\ \frac{\partial f}{\partial t}+v \cdot \nabla f=\left (\frac{d}{dt}f(t, G(t,\cdot))\right )\circ G^{-1}.
	\end{align}
	\par
	Indeed, the material derivative is identical to the pushforward vector field on $Q_T$ given by
	\begin{align}\label{05082030}
		\partial^\bullet\ {=}\  \mathrm{d}\tilde{G}(\frac{\partial}{\partial t}).
	\end{align}
	The following transport formula for time-dependent surface integrals is proved in \cite{DE}.
	\begin{theorem}\label{tf}
		Let $\{\Gamma_t\}_{t\in[0,T]}$ be the family of evolving hypersurfaces described above. Assume that $f$ is a function on $Q_T$ such that all the following quantities exist. Then
		$$\frac{d}{dt}\int_{\Gamma_t}f \mathrm{dvol}_{\Gamma_t}= \int_{\Gamma_t} \big(\partial^\bullet f+ f\nabla_{\Gamma_t}\cdot v \big)\mathrm{dvol}_{\Gamma_t}$$
	\end{theorem}
	As a corollary, we can obtain the following Leibniz formula,
	\begin{corollary}\label{lf}
		For $\varphi,\psi\in C^1(Q_T)$, we have
		\begin{align*}
			\frac{d}{dt}\int_{\Gamma_t}\varphi\psi \mathrm{dvol}_{\Gamma_t}= \int_{\Gamma_t} \varphi\partial^\bullet \psi\mathrm{dvol}_{\Gamma_t}+\int_{\Gamma_t} \psi\partial^\bullet \varphi\mathrm{dvol}_{\Gamma_t} +\int_{\Gamma_t}\varphi\psi\nabla_{\Gamma_t}\cdot v \mathrm{dvol}_{\Gamma_t}
		\end{align*}
	\end{corollary}
	\vskip 0.3cm
	For the proof of Theorem \ref{tf} and Corollary \ref{lf}, we refer the readers to Theorem 5.1 and Lemma 5.2 in \cite{DE} for details.\par
	Through the time-dependent diffeomorphism $G(t,\cdot)$, we introduce a mapping  from $H^{1}({\Gamma_t})$ into $H^{1}({\Gamma_0})$, by defining $\widetilde{u}^t(y)\dot{=}u(G(t,y))$ for $u\in H^{1}({\Gamma_t})$,  $y\in \Gamma_0$. Similarly, for $v\in H^1(\Gamma_0)$, we can also define $\bar{v}^t(x)\ \dot{=}\ v(G^{-1}(t,x))\in H^{1}({\Gamma_t})$.  Moreover, we can easily see that the maps $\widetilde{\ \cdot\ }^t:H^1(\Gamma_t)\rightarrow H^1(\Gamma_0)$ are isomorphisms, with their inverse operator $\bar{\ \cdot\ }^t:H^1(\Gamma_0)\rightarrow H^1(\Gamma_t)$, and their operator norms  can be uniformly bounded by a constant independent of $t$. From (\ref{05082030}), we observe that for $\phi\in C^1(Q_T)$,
	\begin{align}\label{05082036}
		\widetilde{\partial^\bullet\phi}^s(s,y)=\Big(\frac{d}{dt} \widetilde{\phi}^s\Big)(s,y).
	\end{align}
	Before the end of this subsection, we will present the following Poincar\'e inequality and Sobolev's embedding. We also  denote by $[\phi]_t$ the average of $\phi\in L^1(\Gamma_t)$, expressed as
	$$[\phi]_t=\frac{1}{|\Gamma_t|}\int_{\Gamma_t}\phi(x)\mathrm{dvol_{\Gamma_t}(x)},$$
	where we denote  {Vol}$_{\Gamma_t}(\Gamma_t)$ by $|\Gamma_t|$.
	\begin{proposition}\label{poincare}
		Let $\{\Gamma_t\}_{t\in[0,T]}$ be the family of moving hypersurfaces described above, we have the following results:\\
		(i) There exists a positive constant c, such that for every $t\in[0,T]$ and $f\in H^{1}(\Gamma_t)$ , we have the inequality
		$$\|f-[f]_t\|_{L^2(\Gamma_t)}\leq c\|\nabla_{\Gamma_t}f\|_{L^2(\Gamma_t)}.$$
		(ii) There exists a positive constant $C$, such that for every $t\in[0,T]$, $f\in H^1(\Gamma_t)$, $q\leq\frac{2n}{n-2}$ and $q<+\infty$, we have
		$$\|f\|_{L^q(\Gamma_t)} \leq C\|f\|_{H^1(\Gamma_t)}.$$
	\end{proposition}
	\noindent {\bf Proof}. For the Poincar\'e inequality (i), we refer the readers to Theorem 2.12 in \cite{DE} for the proof. The constant $c$ can be made independent of $t$, we refer the readers to Corollary 5.8 in \cite{L} or Corollary A.1.2 in \cite{J} for details. For the Sobolev embedding (ii), by the compactness of $\Gamma_t$, the Sobolev embedding theorem in the Euclidean domain  and (\ref{08010005}) , we have for any $\alpha\leq N$,
	\begin{align*}
		&\Big(\int_{{U^\alpha}}\chi^\alpha(X^\alpha(\theta))|f|^q(X^\alpha(\theta))\sqrt{g^\alpha(t,\theta)}d\theta\Big)^{\frac{1}{q}}\\ \leq\ & C\Big(\int_{{supp\chi^\alpha\circ X^\alpha}}|f|^q(X^\alpha(\theta))d\theta\Big)^{\frac{1}{q}}\\ \leq\ & C\Big(\int_{{{supp\chi^\alpha\circ X^\alpha}}} |f\circ X^\alpha(\theta)|^2d\theta+\int_{{supp\chi^\alpha\circ X^\alpha}}|\frac{\partial f\circ X^\alpha}{\partial \theta_i}|^2(\theta)d\theta\Big)^\frac{1}{2}\\ \leq\ & C\Big(\int_{{{supp\chi^\alpha}}} |f|^2\mathrm{dvol}_{\Gamma_t}(dx)+\int_{{supp\chi^\alpha}}{|\nabla_{\Gamma_t} f|^2\mathrm{dvol}_{\Gamma_t}(dx)}\Big)^\frac{1}{2}\\ \leq\ &C\|f\|_{H^1(\Gamma_t)},
	\end{align*}
	for a constant $C$ independent of $t$. We finish the proof of (ii) by adding the integrals on ${U^\alpha}$'s together.
	$\hfill\blacksquare$\\
	\vskip 0.3cm
	\ \ Let $\dot{H}^{-1}(\Gamma_t)$ denote the closed subspace of ${H}^{-1}(\Gamma_t)$ with $$_{{H}^{-1}(\Gamma_t)}\langle f,1\rangle_{{H}^{1}(\Gamma_t)}=0,\ f\in H^{-1}(\Gamma_t),$$
	and $\dot{H}^1(\Gamma_t)$ denotes the closed subspace of $H^1(\Gamma_t)$ with zero average as well as $\dot{L}^p(\Gamma_t)$ and $\dot{H}^{2,p}(\Gamma_t)$.
	For $f\in \dot{H}^{-1}(\Gamma_t)$, the following inequality holds by Proposition \ref{poincare},
	\begin{align*}
		\|f\|_{{H}^{-1}(\Gamma_t)}=&\sup_{v\in H^1(\Gamma_t),v\neq 0}\frac{_{{H}^{-1}(\Gamma_t)}\langle f,v\rangle_{{H}^{1}(\Gamma_t)}}{\big(\|v\|_{L^2(\Gamma_t)}^2+\|\nabla_{\Gamma_t} v\|^2_{L^2(\Gamma_t)}\big)^\frac{1}{2}}\\=& \sup_{v\in H^1(\Gamma_t),v\neq 0}\frac{_{{H}^{-1}(\Gamma_t)}\langle f,v\rangle_{{H}^{1}(\Gamma_t)}}{\big(\|v-[v]_t\|_{L^2(\Gamma_t)}^2+\|\nabla_{\Gamma_t} v\|^2_{L^2(\Gamma_t)}\big)^\frac{1}{2}}\\\geq& \sup_{v\in \dot{H}^1(\Gamma_t),v\neq 0}\frac{_{{H}^{-1}(\Gamma_t)}\langle f,v\rangle_{{H}^{1}(\Gamma_t)}}{\|\nabla_{\Gamma_t} v\|_{L^2(\Gamma_t)}}\cdot \frac{1}{\sqrt{c^2+1}}\ .
	\end{align*}
	Define
	$$\|f\|_{\dot{H}^{-1}(\Gamma_t)}\dot{=}\sup_{v\in \dot{H}^1(\Gamma_t),v\neq 0}\frac{_{{H}^{-1}(\Gamma_t)}\langle f,v\rangle_{{H}^{1}(\Gamma_t)}}{\|\nabla_{\Gamma_t} v\|_{L^2(\Gamma_t)}}{=}\sup_{v\in {H}^1(\Gamma_t),v\neq 0}\frac{_{{H}^{-1}(\Gamma_t)}\langle f,v\rangle_{{H}^{1}(\Gamma_t)}}{\|\nabla_{\Gamma_t} v\|_{L^2(\Gamma_t)}}. $$ Then for any $t\in[0,T],\ f\in \dot{H}^{-1}(\Gamma_t)$,  we have
	\begin{align}\label{05070340}
		\|f\|_{\dot{H}^{-1}(\Gamma_t)}\approx \|f\|_{{H}^{-1}(\Gamma_t)}.
	\end{align}
	By (ii) of Proposition \ref{poincare}, we see that for any $t\in[0,T]$, $p\geq\frac{2d}{d+2}$, we have
	\begin{align}
		\dot{L}^p(\Gamma_t)\hookrightarrow \dot{H}^{-1}(\Gamma_t).
	\end{align}
	Later on, we will take $\dot{H}^{-1}(\Gamma_0)$ as the pivot space of the Gelfand triplet $\dot{L}^p(\Gamma_0)\hookrightarrow \dot{H}^{-1}(\Gamma_0)\hookrightarrow [\dot{L}^p(\Gamma_0)]^*$, equipped with the norm $\|\cdot\|_{\dot{H}^{-1}(\Gamma_0)}$, for $p\geq \frac{2n}{n+2}$.
	
	\subsection{Main results of Section 3}
	\ \ \ \  In this subsection, we will introduce the mathematical setting of stochastic Stefan  problem on moving hypersurfaces. The Stefan type problem models the melting or freezing process of ice-water mixture, representing one of the most classical and well-known free boundary problems. It also has applications in many other areas like finance. Several studies, see e.g. \cite{BD,KM,KMS,KZS}, have been curried out for stochastic Stefan problems. In the context of moving domain, we refer the readers to \cite{AE} and \cite{ACEV} for the well-posedness of the deterministic version of the Stefan problem. The well-posedness of the Stefan problem on moving hypersurfaces perturbed by multiplicative random noise is established in this subsection. We want to stress that this is the first time a stochastically forced Stefan problem is considered in the context of moving domains. \par

	Let $Q_T$ be given in (\ref{05120430}). The equation we consider is written as
	\begin{numcases}{}\label{pme}
		\partial^\bullet X_t+X_t\nabla_{\Gamma_t}\cdot v(t)=\Delta_{\Gamma_t}\Psi(X_t)+B(t,X_t)\partial^\bullet B_t,\ \text{on}\ Q_T
		\\
		X_0=x_0\in \dot{H}^{-1}(\Gamma_0),\nonumber
	\end{numcases}
	where the function $\Psi(\cdot):\mathbb{R}\rightarrow\mathbb{R}$ satisfies the following conditions
	\begin{itemize}
		\item [\hypertarget{$\Psi$1}{{\bf ($\Psi$1)}}] $\Psi$ is continuous,
		\item [\hypertarget{$\Psi$2}{{\bf ($\Psi$2)}}]  $\forall$ $s,t\in\mathbb{R}$, $(t-s)\big(\Psi(t)-\Psi(s)\big)\geq 0$,
		\item [\hypertarget{$\Psi$3}{{\bf ($\Psi$3)}}] There exists $p>1$, $a>0$, $c_4\geq 0$ such that $\forall s\in\mathbb{R}$, $$s\Psi(s)\geq a|s|^p-c_4,$$
		\item [\hypertarget{$\Psi$4}{{\bf ($\Psi$4)}}] There exist positive constants $c_5$ and $c_6$ such that for $s\in\mathbb{R}$,
		$|\Psi(s)|\leq c_5+c_6|s|^{p-1}$.
	\end{itemize}
	It can be easily seen that equation (\ref{sfe}) is of the form (\ref{pme}) with $\Psi(\cdot)=\beta(\cdot)$ given by
	\begin{numcases}{\beta(r)=}\label{05092308}
		\nonumber ar,\ \ r<0,\\
		0,\ \ 0\leq r\leq \rho,\\
		\nonumber b(r-\rho),\ \ r>\rho,
	\end{numcases}
	for some $a,b,\rho>0$. We remark that the porous media equations $(\Psi(s)=|s|^{p-2}s)$ satisfy the above conditions. The main purpose of this section is to establish the well-posedness of (\ref{pme}) with $p$ in an interval. \par
	As in Subsection 3.2, we can map elements from ${H}^{-1}(\Gamma_t)$ to ${H}^{-1}(\Gamma_0)$. For $f\in H^{-1}(\Gamma_t)$, we define $\widetilde{f}^{t,*}\in H^{-1}(\Gamma_0)$ as follows,
	$$\ _{H^{-1}(\Gamma_0)}\langle\widetilde{f}^{t,*},u\rangle_{H^{1}(\Gamma_0)}\ \dot{=}\ _{H^{-1}(\Gamma_t)}\langle f,\bar{u}^t\rangle_{H^{1}(\Gamma_t)},\ \text{for $u\in H^1(\Gamma_0)$.}$$ \par
	Similarly, we can define $\bar{g}^{t,*}$ for $g\in H^{-1}(\Gamma_0)$. We can also show that the operators $\widetilde{\ \cdot \ }^{t,*}:H^{-1}(\Gamma_t)\rightarrow H^{-1}(\Gamma_0)$ are isomorphisms, with the inverse operator $\bar{\ \cdot\ }^{t,*}:H^{-1}(\Gamma_0)\rightarrow H^{-1}(\Gamma_t)$, and their operator norms  can be uniformly bounded by a constant independent of $t\in[0,T]$.\par
	Notably, if $f\in \dot{H}^{-1}(\Gamma_t)$, we observe that $\widetilde{f}^{t,*}\in \dot{H}^{-1}(\Gamma_0)$ since $\ _{H^{-1}(\Gamma_0)}\langle\widetilde{f}^{t,*},1\rangle_{H^1(\Gamma_0)}$ $=\ _{H^{-1}(\Gamma_t)}\langle f,1\rangle_{H^1(\Gamma_t)}=0$. But  $\widetilde f^{t}\notin \dot{H}^1(\Gamma_0)$ in general for $f\in \dot{H}^1(\Gamma_t)$.\par
	Let us begin by introducing necessary function spaces on time-dependent domains. Set
	\begin{align*}
		C\big([0,T],\dot{H}^{-1}(\Gamma_\cdot)\big)\dot{=}&\big\{\{u(t)\}_{t\in[0,T]}:u(t)\in \dot{H}^{-1}(\Gamma_t) \text{ for any $t\in[0,T]$} \\&\ \text{such that $\widetilde{u(\cdot)}^{\cdot,*}$ }\in C([0,T],\dot{H}^{-1}\big(\Gamma_0)\big)  \big\},
	\end{align*}
	\vskip -0.8cm
	\begin{align*}
		L^p\big([0,T],\dot{L}^p(\Gamma_\cdot)\big)\dot{=}&\big\{\{u(t)\}_{t\in[0,T]}:u(t)\in \dot{L}^{p}(\Gamma_t) \text{ for $a.e.$ $t\in[0,T]$}, \\&\ \text{such that $\widetilde{u(\cdot)}^{\cdot,*}$ }\in L^p([0,T],\dot{L}^p(\Gamma_0)\big)  \big\}.
	\end{align*}
	Consider the Radon-Nikodym derivative of the volume measure generated by $g^t$ with respect to that generated by $g^0=g$, locally defined by
	\begin{align}\label{05092313}
		\frac{\mathrm{dvol}_{g^t}}{\mathrm{dvol}_g}\big(X^\alpha(\theta)\big)=\sqrt{\frac{g^\alpha(t,\theta)}{g^\alpha(\theta)}},\ \text{ $\theta\in U^\alpha$}.
	\end{align}
	It belongs to $C^1([0,T]\times\Gamma_0)$ as well as its inverse $\frac{\mathrm{dvol}_{g}}{\mathrm{dvol}_{g^t}}$. Indeed, we can show
	\begin{align}\label{05101834}
		\frac{\mathrm{dvol}_{g^t}}{\mathrm{dvol}_{g}}(y)\approx 1.
	\end{align}
	\begin{remark}\label{05112043}
		For $p\geq\frac{2n}{n+2}$, if $f\in \dot{L}^p(\Gamma_t)\subseteq \dot{H}^{-1}(\Gamma_t)$, then we have  $\widetilde{f}^{t,*}(y)=f(G(t,y))\frac{\mathrm{dvol}_{g^t}}{\mathrm{dvol}_{g}}(y)\in \dot{L}^p(\Gamma_0)$.
	\end{remark}
	Also, we say a family $\{u(t)\}_{t\in[0,T]}$ an $\dot{H}^{-1}(\Gamma_\cdot)$-valued random process if for any $t\in[0,T]$, $u(t)$ is an $\dot{H}^{-1}(\Gamma_t)$-valued random variable. Concerning predictability,  an $\dot{H}^{-1}(\Gamma_\cdot)$-valued random process $\{u(t)\}_{t\in[0,T]}$ is said to be predictable with respect to the filtration $\{\mathcal{F}_t\}_{t\in[0,T]} $ if the transformed process $\widetilde{u(\cdot)}^{\cdot,*}$ is an $\dot{H}^{-1}({\Gamma_0})$-valued predictable process. Similarly, we can define predictability for $H^1(\Gamma_\cdot)$-valued stochastic process by the operators $\widetilde{\ \cdot \ }^{t}:H^{1}(\Gamma_t)\rightarrow H^{1}(\Gamma_0)$ and $\bar{\ \cdot \ }^{t}:H^{1}(\Gamma_0)\rightarrow H^{1}(\Gamma_t)$. \par Having defined predictability, we can now describe the stochastic term in (\ref{pme}). In ($\ref{pme}$), we assume that the Wiener noise $B$ is an $l^2$ cylindrical Wiener process represented by $\{\beta^k\}_{k\geq 1}$. The diffusion coefficient $\sigma(\cdot,\cdot)$ is represented as a family of mappings $\big\{\{\sigma_k(t,\cdot):\dot{H}^{-1}(\Gamma_t)\rightarrow \dot{H}^{-1}(\Gamma_t)\}_{t\in[0,T]}\big\}_{k\geq 1}$ such that the transformed diffusion coefficient
	\begin{align}\label{05070010}
		\widetilde{\sigma}_k(\cdot ,\cdot):[0,T]\times \dot{H}^{-1}(\Gamma_0)\rightarrow \dot{H}^{-1}(\Gamma_0),
	\end{align}
	defined by
	\begin{align}\label{06212153}
		\widetilde{\sigma}_k(t ,v)\ \dot{=}\ \widetilde{\sigma}_k^{t,*}(t,\bar{v}^{t,*}),
	\end{align}
	for any $(t,v)\in[0,T]\times \dot{H}^{-1}(\Gamma_0)$,
	is measurable for every $k\geq 1$. Consequently, for any $\dot{H}^{-1}(\Gamma_0)$-valued predictable process $v(\cdot)$,
	the $\dot{H}^{-1}(\Gamma_\cdot)$-valued random process
	$$t\rightarrow \widetilde{\sigma}_k(t,{v(t)})={\sigma}_k(t,{\bar{v}^{t,*}(t)})\in \dot{H}^{-1}(\Gamma_\cdot)\text{ is predictable.} $$
	Then, for any $H^1(\Gamma_\cdot)$-valued predictable process $\phi$ and any $\dot{H}^{-1}(\Gamma_\cdot)$-valued predictable process ${u}(\cdot)$, the real-valued process
	$$t\rightarrow_{H^{-1}(\Gamma_t)}\langle \sigma_k(t,u(t)), \phi(t)\rangle_{H^1(\Gamma_t)}=_{H^{-1}(\Gamma_0)}\langle \widetilde{\sigma}^t_k(t,\widetilde{u}^{t,*}(t)), \widetilde{\phi}^t(t)\rangle_{H^1(\Gamma_0)} $$
	is predictable. Thus, with appropriate conditions imposed on $\sigma_k$, $u$ and $\phi$, the stochastic integral
	\begin{align}\label{05070036}
		\int_{0}^{t}\ _{H^{-1}(\Gamma_s)}\langle \sigma_k(s,u(s)),\phi(s)\rangle_{H^1(\Gamma_s)}d\beta^k_s
	\end{align}
	is well-defined.
	To solve equation (\ref{pme}), we introduce linear growth condition and Lipschitz condition on the diffusion coefficients $\{\sigma_k\}_{k\geq 1}$. Let's assume there exists $f\in L^1([0,T],\mathbb{R}_+)$ such that for a.e. $t\in[0,T]$, and for any $u,v\in \dot{H}^{-1}(\Gamma_t)$,
	\begin{align}
		&\sum_{k=1}^{\infty}\|\sigma_k(t,u)-\sigma_k(t,v)\|^2_{\dot{H}^{-1}(\Gamma_t)}\leq f(t)\|u-v\|^2_{\dot{H}^{-1}(\Gamma_t)}\label{lp},\\& \sum_{k=1}^{\infty}\|\sigma_k(t,u)\|_{\dot{H}^{-1}(\Gamma_t)}^2\leq f(t)(\|u\|^2_{\dot{H}^{-1}(\Gamma_t)}+1).\label{lg}
	\end{align}
	Now we can give a precise definition of solutions to equation (\ref{pme}). Building upon Corollary \ref{lf}, we give a weak formulation of the stochastic equations on moving hypersurfaces.
	\begin{definition}\label{ws}
		Let $x_0\in\dot{H}^{-1}(\Gamma_0)$, $p>1$. We say an $\dot{H}^{-1}(\Gamma_\cdot)$-valued random process $\{u(t)\}_{t\in[0,T]}$ is a solution to (\ref{pme}) if
		\begin{itemize}{}
			\item [(i)]: $u\in C([0,T], \dot{H}^{-1}\big(\Gamma_\cdot)\big)\cap L^p\big([0,T],\dot{L}^p(\Gamma_\cdot)\big)$ $a.e.$ and
			$$ E\Big[\sup_{t\in[0,T]}\|u(t)\|^2_{\dot{H}^{-1}(\Gamma_t)}+\int_{0}^{T}\|u(t)\|_{L^p(\Gamma_t)}^pdt\Big]<\infty $$
			\item [(ii)]: $u$ is an $\dot{H}^{-1}({\Gamma_\cdot})$-valued predictable random process with regard to the filtration $\{\mathcal{F}_t\}_{t\in[0,T]} $.
			\item[(iii)]: For any $\phi\in C^2(Q_T)$ and $t\in[0,T]$, we have
			\begin{align}\label{05082042}
				&\nonumber\ \ \ _{H^{-1}(\Gamma_t)}\langle u(t),\phi(t)\rangle_{H^{1}(\Gamma_t)}-_{H^{-1}(\Gamma_0)}\langle x_0,\phi(0)\rangle_{H^{1}(\Gamma_0)} \\&\nonumber=\int_{0}^{t}\int_{{\Gamma_s}} \Psi(u(s,x))\Delta_{\Gamma_s}\phi(s,x)\mathrm{dvol_{\Gamma_s}}(dx)ds+\int_{0}^{t}\ _{H^{-1}(\Gamma_s)}\langle u(s),\partial^\bullet\phi(s)\rangle_{H^{1}(\Gamma_s)}ds\\
				&+ \sum_{k=1}^{\infty}\int_{0}^{t}\ _{H^{-1}(\Gamma_s)}\langle \sigma_k(s,u(s)),\phi(s)\rangle_{H^1(\Gamma_s)}d\beta^k_s.
			\end{align}
		\end{itemize}
	\end{definition}
	\begin{remark}
		The weak formulation (iii) gives a mathematical explanation why zero average is preserved despite the evolving of the hypersurfaces. Indeed, setting $\phi\equiv 1$ in (iii), we obtain $_{H^{-1}({\Gamma_t})}\langle u(t),1\rangle_{H^1({\Gamma_t})}= _{H^{-1}({\Gamma_0})}\langle x_0,1\rangle_{H^1({\Gamma_0})}$ for any $t\in[0,T]$. Thus, the condition $x_0\in\dot{H}^{-1}(\Gamma_0)$ implies $u(t)\in \dot{H}^{-1}(\Gamma_t)$ for any $t\in[0,T]$ .
	\end{remark}
	Now we can state the main result of Section 3.
	\begin{theorem}\label{main result}
		Assuming $x_0\in \dot{H}^{-1}(\Gamma_0)$, the function $\Psi$ in equation $(\ref{pme})$ satisfies (\hyperlink{$\Psi$1}{$\Psi$1})-(\hyperlink{$\Psi$4}{$\Psi$4}) with $p\in[\frac{2n}{n+2},\infty)$ if $n=1,2$ and $p\in[\frac{2n}{n+2},\frac{2n}{n-2}]$ if $n\geq3$. Then the equation (\ref{pme}) has a unique solution.
	\end{theorem}
	\begin{proof}
		We will apply Theorem 2.8. The proof is divided into three parts, which are given in the next three subsections. We first verify conditions (\hyperlink{C1}{C1})-(\hyperlink{C4}{C4}) in Subsection 3.4. Then, through transformation we will show that the existence and uniqueness of a solution to equation (\ref{pme}) is equivalent to the existence and uniqueness of a solution to the SPDE (\ref{05100037}). The proof of this correspondence is given in Subsection 3.5. At last, in Subsection 3.6, we show that the SPDE (\ref{05100037}) is well-posed by further verifying conditions (\hyperlink{H1}{H1})-(\hyperlink{H5}{H5}) given in Section 2.3.
	\end{proof}
	\vskip 0.5cm
	\ \ \ As a by-product of the well-posedness, we obtain an It$\mathrm{\hat{o}}$ type formula for the $\|\cdot\|_{\dot{H}^{-1}(\Gamma_t)}$ norm of the solution to equation (\ref{pme}). One can see how the evolution of the hypersurface affects the $\dot{H}^{-1}$-norm of the solution from this formula. Recall the definition of the speed $v$ in $(\ref{08270324})$, we have
	\begin{theorem}
		Under the same assumptions as in Theorem \ref{main result}, the solution $X$ to (\ref{pme}) satisfies the following energy equality $P-a.s.$,
		\begin{align*}
			\nonumber&\ |X_t|_{\dot{H}^{-1}(\Gamma_t)}^2\\=&\nonumber\ |X_0|_{\dot{H}^{-1}(\Gamma_0)}^2-2\int_{0}^{t}\int_{{\Gamma_s}}\Psi\big(X_s(x)\big)X_s(x)\mathrm{dvol}_{\Gamma_s}(dx)ds\\&\ +2\sum_{k=1}^{\infty}\int_{0}^{t}\big(\sigma_k(s,X_s),X_s\big)_{\dot{H}^{-1}(\Gamma_s)}d\beta^k_s +\sum_{k=1}^{\infty}\int_{0}^{t}|{\sigma}_k(s,Y_s)|_{\dot{H}^{-1}(\Gamma_s)}^2ds\\&\ -\int_{0}^{t}\int_{{\Gamma_s}}\mathcal{B}(s,v)\nabla_{\Gamma_s}(-\Delta_{\Gamma_s})^{-1}X_s\cdot \nabla_{\Gamma_s}(-\Delta_{\Gamma_s})^{-1}X_s\mathrm{dvol}_{\Gamma_s}(dx)ds,
		\end{align*}
		where $\mathcal{B}(s,v)$ is the deformation tensor, which will be given in Subsection 3.7 in detail.
	\end{theorem}
	The proof is given in Subsection 3.7.
	\subsection{ Verification of (\hyperref[C1]{C1})-(\hyperref[C4]{C4}) }
	\ \ \ \ In this subsection, we will construct a family of equivalent norms that associated  with the function spaces defined on time-dependent hypersurfaces $\{\Gamma_t\}_{t\in[0,T]}$ and we will verify the conditions (\hyperlink{C1}{C1})-(\hyperlink{C4}{C4}).
	
	According to Proposition \ref{poincare},  $\dot{H}^1(\Gamma_0)$ can be endowed with the norm $\|\nabla_{\Gamma_0}(\cdot)\|_{L^2(\Gamma_0)}$. Consequently, $\dot{H}^{-1}(\Gamma_0)$ is isometrically isomorphic to $\big(\dot{H}^{1}(\Gamma_0)\big)^*$. Let's consider the Gelfand triple $V\dot{=}\ \dot{L}^p(\Gamma_0)\hookrightarrow H\dot{=}\ \dot{H}^{-1}(\Gamma_0)\hookrightarrow V^*\dot{=}\ [\dot{L}^p(\Gamma_0)]^*$, where $V$ is reflexive and  densely embedded into $H$ under the assumption that $p\in[\frac{2n}{n+2},\frac{2n}{n-2}]$ if $n\geq3$ or $p\in[\frac{2n}{n+2},+\infty)$ if $n$=1 or 2. \par Now define the Laplace  operator $(-\Delta_{\Gamma_0}):\dot{H}^1(\Gamma_0)\rightarrow\dot{H}^{-1}(\Gamma_0)$ such that for $u,v\in \dot{H}^1(\Gamma_0)$,
	$$_{H^{-1}(\Gamma_0)}\langle -\Delta_{\Gamma_0} u,v\rangle_{H^{1}(\Gamma_0)}=\int_{\Gamma_0}\nabla_{\Gamma_0} u(y)\cdot \nabla_{\Gamma_0}v(y)\ \mathrm{dvol}_{\Gamma_0}(dy)=(u,v)_{\dot{H}^1(\Gamma_0)}.$$
	This implies that $(-\Delta_{\Gamma_0}):\dot{H}^1(\Gamma_0)\rightarrow\dot{H}^{-1}(\Gamma_0)$ is the Riesz dual operator for the inner product $(\cdot,\cdot)_{\dot{H}^1(\Gamma_0)}$, which is an isometric isomorphism. Then, for any $f,g\ \in \dot{H}^{-1}(\Gamma_0)$,  an inner product $(\cdot ,\cdot )_0$ can be constructed on $\dot{H}^{-1}(\Gamma_0)$, represented by
	$$(f,g)_0= \Big((-\Delta_{\Gamma_0})^{-1}f,(-\Delta_{\Gamma_0})^{-1}g\Big)_{\dot{H}^1(\Gamma_0)}=_{{H}^{-1}(\Gamma_0)}\big\langle f, (-\Delta_{\Gamma_0})^{-1}g\big\rangle_{{H}^{1}(\Gamma_0)},$$ such that for any $f\in \dot{H}^1(\Gamma_0)$,
	$$ |f|_0\dot{=}\sqrt{(f,f)_0}=\sup_{u\in \dot{H}^1(\Gamma_0)}\frac{_{H^{-1}(\Gamma_0)}\langle f,u\rangle_{H^1(\Gamma_0)}}{\|u\|_{\dot{H}^1(\Gamma_0)}}.$$
	Moreover, for $t\geq 0$,  we can create a new inner product on  $\dot{H}^1(\Gamma_0)$ denoted as $(\cdot\ ,\ \cdot)_{\dot{H}^1(\Gamma_0^t)}$, such that for $u,v\in \dot{H}^1(\Gamma_0)$
	\begin{align}\label{05091809}
		(u,v)_{\dot{H}^1(\Gamma_0^t)}&\dot{=}\sum_{\alpha=1}^{N}\int_{{U^\alpha}}\chi ^\alpha(X^\alpha(\theta))g^{ij}_{t,\alpha}(\theta)\frac{\partial u}{\partial \theta_i}\frac{\partial v}{\partial \theta_j}\sqrt{g^\alpha(t,\theta)}d\theta\nonumber\\&\Big(=\int_{\Gamma_0}\big(\mathrm{grad}^t u,\mathrm{grad}^t v)_{g^t}\mathrm{dvol}_{g^t}\Big),
	\end{align}
	where $g^{t}$ is the pullback Riemannian metric from $\Gamma_t$ to $\Gamma_0$, locally defined by (\ref{05082113}), and $\mathrm{grad}^t u$ denotes the gradient of $u$ with respect to  the Riemannian metric $g^{t}$, whose definition can be seen e.g.  in (3.1.16) and (3.1.17) in \cite{J}. We also denote the norm corresponding to the inner product (\ref{05091809}) by $\|\cdot\|_{\dot{H}^1(\Gamma_0^t)}$.\par
	For any $t\in[0,T]$ and $u\in {H}^1(\Gamma_t)$, we can verify that
	\begin{align}\label{05111935}
		\|\widetilde{u}^t\|_{\dot{H}^1(\Gamma^t_0)}=\|u\|_{\dot{H}^1(\Gamma_t)}.
	\end{align}
	Using the  local chart, by the compactness of the support of $\chi^\alpha$ and the compactness of $\Gamma_0$, we can also deduce that for any $u\in\dot{H}^1(\Gamma_0)$, $t\in[0,T]$,
	\begin{align}\label{05072145}
		\|u\|_{\dot{H}^1(\Gamma_0)}\approx\|u\|_{\dot{H}^1(\Gamma^t_0)}.
	\end{align}
	Due to (\ref{05072145}), we can define a family of equivalent norms $\{|\cdot|_{t}\}_{t\in[0,T]}$ on the space $\dot{H}^{-1}(\Gamma_0)$ given by
	\begin{align}\label{06220053}
		|f|_t\dot{=}\sup_{u\in \dot{H}^1(\Gamma_0)}\frac{_{H^{-1}(\Gamma_0)}\langle f,u\rangle_{H^{1}(\Gamma_0)}}{\|u\|_{\dot{H}^1(\Gamma_0^t)}}.
	\end{align}
	Moreover, for $f\in\dot{H}^{-1}(\Gamma_t)$, we infer from (\ref{05111935}),
	\vskip -0.7cm
	\begin{align}\label{05082011}
		\nonumber&\|f\|_{\dot{H}^{-1}(\Gamma_t)}=\sup_{v\in H^1(\Gamma_t),v\neq 0}\frac{_{{H}^{-1}(\Gamma_t)}\langle f,v\rangle_{{H}^{1}(\Gamma_t)}}{\centering{\|\nabla_{\Gamma_t}v\|}_{L^2(\Gamma_t)}}= \sup_{v\in H^1(\Gamma_t),v\neq 0}\frac{_{{H}^{-1}(\Gamma_0)}\langle \widetilde{f}^{t,*},\widetilde{v}^t\rangle_{{H}^{1}(\Gamma_0)}}{\centering{\|\nabla_{\Gamma_t}v\|}_{L^2(\Gamma_t)}}\\=& \sup_{v\in H^1(\Gamma_t),v\neq 0}\frac{_{{H}^{-1}(\Gamma_0)}\langle \widetilde{f}^{t,*},\widetilde{v}^t\rangle_{{H}^{1}(\Gamma_0)}}{\centering{\|\mathrm{grad}^t \widetilde{v}^t\|}_{L^2(\mathrm{dvol}_{g^t})}}= \sup_{u\in {H}^1(\Gamma_0),u\neq 0}\frac{_{{H}^{-1}(\Gamma_0)}\langle \widetilde{f}^{t,*},u\rangle_{{H}^{1}(\Gamma_0)}}{\centering{\|u\|}_{\dot{H}^1(\Gamma_0^t)}}=|\widetilde{f}^{t,*}|_t.
	\end{align}
	We'd like to emphasize that this observation motivates the introduction of the equivalent norms $|\cdot|_t$ on $\dot{H}^{-1}(\Gamma_0)$.\par
	Similarly, we define a family of operators $\{R_t :\dot{H}^1(\Gamma_0)\rightarrow \dot{H}^{-1}(\Gamma_0)\}_{t\in[0,T]}$, denoted by $R_t\ \dot{=}\ (-\frac{\mathrm{dvol}_{g^t}}{\mathrm{dvol}_{g}}\Delta_{g^t}):\dot{H}^1(\Gamma_0)\rightarrow \dot{H}^{-1}(\Gamma_0)$, such that for any $u,v\in \dot{H}^{1}(\Gamma_0) $,
	\begin{align}\label{05091808}
		_{H^{-1}(\Gamma_0)}\langle R_tu,v\rangle_{H^1(\Gamma_0)}= _{H^{-1}(\Gamma_0)}\langle-\frac{\mathrm{dvol}_{g^t}}{\mathrm{dvol}_{g}}\Delta_{g^t}u,v\rangle_{H^{1}(\Gamma_0)}=(u,v)_{\dot{H}^1(\Gamma_0^t)}.
	\end{align}
	As before, this operator serves as the Riesz dual operator for the Hilbert space $\dot{H}^1(\Gamma_0)$ equipped with the inner product $(\cdot,\cdot)_{\dot{H}^1(\Gamma_0^t)}$, which means for any $v\in \dot{H}^1(\Gamma_0)$,
	\begin{align}\label{08262014}
		|R_tv|_{t}=\sup_{u\in {H}^1(\Gamma_0),u\neq 0}\frac{_{{H}^{-1}(\Gamma_0)}\langle R_t v,u\rangle_{{H}^{1}(\Gamma_0)}}{\centering{\|u\|}_{\dot{H}^1(\Gamma_0^t)}}=\|v\|_{\dot{H}^1(\Gamma^t_0)}.
	\end{align}
	Therefore, it is also an isomorphism and we denote its inverse as $R_{-t}:\ \dot{H}^{-1}(\Gamma_0)\rightarrow \dot{H}^1(\Gamma_0)$.  Additionally, $\dot{H}^{-1}(\Gamma_0)$ can  be equipped with an inner product $(\cdot,\cdot)_t$, defined by
	\vskip -0.7cm
	\begin{align*}
		(f,g)_t&\dot{=}\ (R_{-t}f,R_{-t}g)_{\dot{H}^1(\Gamma_0^t)}=_{{H}^{-1}(\Gamma_0)}\langle f, R_{-t}g\rangle_{{H}^1(\Gamma_0)}\\&=_{{H}^{-1}(\Gamma_0)}\langle f,R_{-0} R_0R_{-t}g\rangle_{{H}^1(\Gamma_0)}=(f,R_0R_{-t}g)_0\\&=(f,\iota_t^*g)_0.
	\end{align*}
	where $\iota_t^*:\dot{H}^{-1}(\Gamma_0)\rightarrow\dot{H}^{-1}(\Gamma_0)$, is defined as $\iota_t^*g=R_0R_{-t}g$. By the property of the Riesz dual operator, we have for any $f\in \dot{H}^{-1}(\Gamma_0)$,
	\begin{align}\label{07011926}
		(f,f)_t=|f|^2_t.
	\end{align}
	\par
	\vskip 0.3cm
	\noindent {\bf Verification of (\hyperlink{C1}{C1})}:
	The condition (\hyperlink{C1}{C1}) follows directly from (\ref{05072145}), (\ref{06220053}) and (\ref{07011926}).\\
	\vskip 0.3cm
	\noindent {\bf Verification of (\hyperlink{C2}{C2})}: From (\ref{05091809}), (\ref{05091808}) and the regularity assumptions on $X^\alpha(\cdot)$ and $G(\cdot,\cdot)$, we assert that
	\begin{align}\label{05091824}
		R_\cdot\in C^1\Big([0,T],\mathcal{L}\big(\dot{H}^1(\Gamma_0),\dot{H}^{-1}(\Gamma_0)\big)\Big),
	\end{align}
	with its time-derivative represented as  $\Psi(s):\dot{H}^1(\Gamma_0)\rightarrow \dot{H}^{-1}(\Gamma_0)$, given by,
	\begin{align}\label{05112023}	\Psi(s)v\dot{=}\frac{d}{ds}\big(-\frac{\mathrm{dvol}_{g^s}}{\mathrm{dvol} _g}(y)\Delta_{g^s}v\big)=\frac{d}{ds}\big(R_sv\big),
	\end{align}
	such that for any $v,w\in \dot{H}^{1}(\Gamma_0)$,
	$$_{{H}^{-1}(\Gamma_0)}\langle \Psi(s)v,w\rangle_{{H}^{1}(\Gamma_0)}=\sum_{\alpha=1}^{N}\int_{{U^\alpha}}\chi^\alpha(X^\alpha(\theta))\frac{\partial w}{\partial \theta^i}\frac{\partial v}{\partial \theta^j}\frac{d}{ds}\big(\sqrt{g^\alpha(s,\theta)}g_{s,\alpha}^{ij}(\theta)\big)d\theta.$$
	We can verify that this expression is independent of the choice of local charts.
	
	Due to (\ref{08010005}), the compactness of $supp\chi^\alpha\circ X^\alpha$ in $U^\alpha$, the compactness of $\Gamma_0$, and the regularity of coefficients $\sqrt{g^\alpha(t,\cdot)}g_t^{ij}(\cdot)$, we infer that $\|\Psi(t)\|_{\mathcal{L}(\dot{H}^1(\Gamma_0),\dot{H}^{-1}(\Gamma_0))}\lesssim 1$.
	Let's define
	$$\Phi(s)\dot{=}-R_0 R_{-s}\Psi(s)R_{-s}:\dot{H}^{-1}(\Gamma_0)\rightarrow\dot{H}^{-1}(\Gamma_0). $$
	By (\ref{05091824}) and (\ref{05112023}), we see that
	\begin{align*}
		\iota_\cdot^*\in C^1\Big([0,T],\mathcal{L}\big(\dot{H}^{-1}(\Gamma_0)\big)\Big), with\ \frac{d}{ds}\iota_s^*= \Phi(s).
	\end{align*}
	Moreover, for any $f,g\in \dot{H}^{-1}(\Gamma_0)(=H)$, we have
	\begin{align}\label{05092048}
		\nonumber\big( f,\Phi(s)g\big)_H&=-(f,R_0R_{-s}\Psi(s)R_{-s}g)_H=-_{{H}^{-1}(\Gamma_0)}\langle f,R_{-s}\Psi(s)R_{-s}g\rangle_{{H}^1(\Gamma_0)}\\\nonumber&=-\big( R_{-s}f,R_{-s}\Psi(s)R_{-s}g\big)_{\dot{H}^1(\Gamma_0^s)}=-\ _{{H}^{1}(\Gamma_0)}\langle R_{-s}f,\Psi(s)R_{-s}g\rangle_{{H}^{-1}(\Gamma_0)}\\&=-\sum_{\alpha=1}^{N}\int_{{U^\alpha}}\chi^\alpha(X^\alpha(\theta))\frac{\partial }{\partial \theta^i}[R_{-s}f]\frac{\partial }{\partial \theta^j} [R_{-s}g]\frac{d}{ds}\big(\sqrt{g^\alpha(s,\theta)}g_s^{ij}(\theta)\big)d\theta.
	\end{align}
	This shows that  $\Phi(s):\dot{H}^{-1}(\Gamma_0)\rightarrow \dot{H}^{-1}(\Gamma_0)$ is a self-adjoint operator with $\|\Phi(s)\|_{\mathcal{L}(\dot{H}^{-1}(\Gamma_0))}\lesssim 1$, $0\leq s\leq T$. Furthermore, we have $$\frac{d}{ds}(f,f)_s=\frac{d}{ds}(f,\iota_s^*f)_H=(f,\frac{d}{ds}\iota_s^*f)_H=(f,\Phi(s)f)_{{\dot{H}^{-1}(\Gamma_0)}}.$$
	Consequently, we confirm the validity of $(\hyperlink{C2}{C2})$.\vskip 0.25cm
	\noindent {\bf Verification of (\hyperlink{C3}{C3}) and (\hyperlink{C4}{C4})}:
	Let's take $f\in \dot{L}^p(\Gamma_0)\subseteq \dot{H}^{-1}(\Gamma_0)$.
	By the Riesz representaton theorem, there exist an unique element $v\dot{=}R_{-t}f\in \dot{H}^{1}(\Gamma_0)$ such that for any $u\in \dot{H}^1(\Gamma_0)$
	$$\langle-\frac{\mathrm{dvol}_{g^t}}{\mathrm{dvol}_{g}}\Delta_{g^t}v,u\rangle=(u,v)_{\dot{H}^1(\Gamma_0^t)}=\int_{\Gamma_0}f u \mathrm{dvol}_{\Gamma_0}(y).$$
	By (\ref{05072145}), (\ref{08262014}), (\hyperlink{C1}{C1}) and , we have
	\begin{align}\label{20240508}
		\|v\|_{\dot{H}^1(\Gamma_0)}\approx\|v\|_{\dot{H}^1(\Gamma_0^t)}=|f|_t\approx |f|_0\lesssim |f|_{\dot{L}^p(\Gamma_0)} .
	\end{align}
	Moreover, employing local charts, we observe that locally $v\circ X^\alpha(\theta)$ is a distributional solution of the following elliptic equation on $U^\alpha$,
	\begin{align}\label{05092203}
		-\frac{1}{\sqrt{g^\alpha(\theta)}}\frac{\partial}{\partial \theta_i}\Big(\sqrt{g^\alpha(t,\theta)}g^{ij}_{t,\alpha}(\theta)\frac{\partial v}{\partial \theta_j}(\theta)\Big)=f\big(X^\alpha(\theta)\big).
	\end{align}
	By the local regularity estimates in Theorem 4.2 in Chapter 3 of \cite{CW}, we can find a compact set $K^\alpha$ such that $supp\chi^\alpha\circ X^\alpha\subseteq K^\alpha\subseteq U^\alpha$ and a constant $C^\alpha$ independent of $t$ such that
	\begin{align}\label{05081903}
		\nonumber\|v\circ X^\alpha\|_{W^{2,p}(supp\chi^\alpha\circ X^\alpha)}&\leq C^\alpha\big(\|f\circ X^\alpha\|_{L^p(K^\alpha)}+\|v\circ X^\alpha\|_{L^p(K^\alpha)}\big)\\&\leq C^\alpha\big(\|f\|_{L^p(\Gamma_0)}+\|v\|_{L^p(\Gamma_0)}\big).
	\end{align}
	The second inequality follows from the compactness of $K^\alpha$.
	Additionally, using local charts, the compactness of $\Gamma_0$ and ($\ref{08010005}$), we can deduce that the $H^{2,p}(\Gamma_0)$ norm given in (\ref{05081740}) is equivalent to the norm $\|\cdot\|_{2,p}$ defined by local charts and $\chi^\alpha$,
	\vskip -0.5cm
	\begin{align}\label{05092209}
		\|u\|_{2,p}^p=\sum_{\alpha=1}^N\int_{{U^\alpha}}\chi^\alpha\circ X^\alpha(\theta)\Big(|u|^p(X^\alpha(\theta))+|\nabla \big(u\circ X^\alpha\big)|^p(\theta)+|\nabla^2 \big(u\circ X^\alpha\big)|^p(\theta)\Big)d\theta.
	\end{align}
	Thus, from (\ref{05081903}), (\ref{20240508}) and (ii) of Theorem \ref{poincare}, we infer that
	\begin{align}\label{07020120}
		\|R_{-t}f\|_{{{H}^{2,p}}(\Gamma_0)}=\|v\|_{{H^{2,p}}(\Gamma_0)}\leq C \|f\|_{{L}^p(\Gamma_0)}.
	\end{align}
	Consequently,
	\begin{align*}
		\|\iota_t^*f\|_V=\|\Delta_{\Gamma_0} R_{-t}f\|_{L^p(\Gamma_0)}\leq \|R_{-t}f\|_{{H}^{2,p}(\Gamma_0)}\leq C\|f\|_{\dot{L}^p(\Gamma_0)}=C\|f\|_{V}.
	\end{align*}
	Similarly,  $\|\iota_{-t}^*f\|_V\leq C\|f\|_{V}$, completing the verification of $(\hyperlink{C1}{C1})$-$(\hyperlink{C4}{C4})$.
	\subsection{Equivalence of solutions}
	In this subsection, we will show that the solution of the stochastic Stefan problem  (\ref{pme}) can be transformed to a solution of an SPDE with nonhomogeneous monotonicity in the variational setting and vice versa.
	Let $X$ be a solution to (\ref{pme}) that satisfies Definition \ref{ws}. Let $\phi\in C^2(Q_T)$. Denote $\widetilde{X}^{s,*}_s$ by $Y_s$. We will show that $\{Y_s, s\geq 0\}$ is a solution to a SPDEs in a Gelfand triple: $\dot{L}^p(\Gamma_0)\hookrightarrow \dot{H}^{-1}(\Gamma_0)\hookrightarrow\big(\dot{L}^p(\Gamma_0)\big)^*$.   By Remark \ref{05112043}, we have $Y_s(y)=X_s(G(s,y))\frac{\mathrm{dvol}_{g^s}}{\mathrm{dvol}_{g}}(y)\in \dot{L}^p(\Gamma_0)$ if $X_s\in \dot{L}^p(\Gamma_s)$. We also let $\widetilde{\phi}:[0,T]\times\Gamma_0\rightarrow \mathbb{R}$ be defined as $\widetilde{\phi}(s,y)=\phi\big(s,G(s,y)\big)=\widetilde{\phi}^s(y)$. Using local charts, we can show that for $\phi\in C^2(Q_T)$,
	\begin{align}\label{05100247}
		\widetilde{\Delta_{\Gamma_s} \phi(s)}^{s,*}=\Delta_{g^s}\widetilde{\phi(s)}^s,
	\end{align}
	where $\Delta_{g^s}$ is the Laplacian-Beltrami operator on $\Gamma_0$ associated with the Riemannian metric $g^s$, defined locally by ($\ref{05082113}$).
	
	By (\ref{05082036}), in terms of $Y, \widetilde{\phi}$,   the weak formulation (\ref{05082042}) reads as
	\begin{align}\label{05082045}
		&\nonumber\ \ \ _{H^{-1}(\Gamma_0)}\langle Y_t,\widetilde{\phi}(t)\rangle_{H^{1}(\Gamma_0)}-_{H^{-1}(\Gamma_0)}\langle x_0,\widetilde{\phi}(0)\rangle_{H^{1}(\Gamma_0)} \\&\nonumber=\int_{0}^{t}\int_{{\Gamma_0}} \Psi\big( Y_s(y)\frac{\mathrm{dvol}_{g}}{\mathrm{dvol}_{g^s}}(y)\big)\Delta_{g^s}\widetilde{\phi}(s,y)\mathrm{dvol}_{g^s}(dy)ds+\int_{0}^{t}\ _{H^{-1}(\Gamma_0)}\langle Y_s,\frac{d}{dt}\widetilde{\phi}(s)\rangle_{H^{1}(\Gamma_0)}ds\\
		&+ \sum_{k=1}^{\infty}\int_{0}^{t}\ _{H^{-1}(\Gamma_0)}\langle \widetilde{\sigma}_k(s,Y_s),\widetilde{\phi}(s)\rangle_{H^1(\Gamma_0)}d\beta^k_s.
	\end{align}
	For any $\widetilde{\phi}\in C^2([0,T]\times\Gamma_0)$, writing
	\begin{align}\label{05092136}
		\varphi(s)\dot{=}R_s\big(\widetilde{\phi}(s)-[\widetilde{\phi}(s)]_0\big)\in C^1([0,T],\dot{H}^{-1}(\Gamma_0)),
	\end{align}
	we have  $\widetilde{\phi}(s)-[\widetilde{\phi}(s)]_0=R_{-s}\varphi(s)$. Then,
	\begin{align*}
		\frac{d}{ds}\Big(\widetilde{\phi}(s)-[\widetilde{\phi}(s)]_0\Big)=- R_{-s}\frac{dR_s}{ds}R_{-s}\varphi(s)+R_{-s}\dot{\varphi}(s)\in C\big([0,T],\dot{H}^1(\Gamma_0)\big).
	\end{align*}
	Since $Y_s\in \dot{H}^{-1}(\Gamma_0)$ , we have by (\ref{05092048}), \vskip -0.6cm
	\begin{align}\label{05112119}
		\nonumber&_{{H}^{-1}(\Gamma_0)}\langle Y_s,\frac{d}{ds}\widetilde{\phi(s)}\rangle_{{H}^{1}(\Gamma_0)}=	_{{H}^{-1}(\Gamma_0)}\langle Y_s,\frac{d}{ds}\big(\widetilde{\phi(s)}-[\widetilde{\phi(s)}]_0\big)\rangle_{{H}^{1}(\Gamma_0)}\\=\nonumber& -	_{{H}^{-1}(\Gamma_0)}\langle Y_s,R_{-s}\frac{dR_s}{ds}R_{-s}\varphi(s)\rangle_{{H}^{1}(\Gamma_0)}+_{{H}^{-1}(\Gamma_0)}\langle Y_s,R_{-s}\dot{\varphi}(s)\rangle_{{H}^{1}(\Gamma_0)}\\=& \big( Y_s,\Phi(s)\varphi(s)\big)+\big( Y_s,\dot{\varphi}(s)\big)_s.
	\end{align}
	Similarly, we can show
	\vskip -0.7cm
	\begin{align}
		\label{05100141}&_{{H}^{-1}(\Gamma_0)}\langle Y_t,\tilde{\phi}(t)\rangle_{{H}^{1}(\Gamma_0)}=\big(Y_t,\varphi(t)\big)_t, \\\label{05100142}&\ _{H^{-1}(\Gamma_0)}\langle \sigma_k(s,Y_s),\tilde{\phi}(s)\rangle_{H^1(\Gamma_0)}=\big(\tilde{\sigma}_k(s,Y_s),{\varphi}(s)\big)_s, \\&\ \label{05100143}\int_{{\Gamma_0}} \Psi\big(Y_s(y)\frac{\mathrm{dvol}_{g}}{\mathrm{dvol}_{g^s}}(y)\big)\Delta_{g^s}\widetilde{\phi}(s,y)\mathrm{dvol}_{g^s}(dy)\nonumber\\=&\ -\int_{{\Gamma_0}} \Psi\big(Y_s(y)\frac{\mathrm{dvol}_{g}}{\mathrm{dvol}_{g^s}}(y)\big)\varphi(s,y)\mathrm{dvol}_{\Gamma_0}(dy).
	\end{align}
	Combining (\ref{05112119}) and (\ref{05100141})-(\ref{05100143}) together, we obtain from (\ref{05082045}) that
	\begin{align}\label{05092236}
		\nonumber&\big(Y_t,\varphi(t)\big)_t-\big(Y_0,\varphi(0)\big)_0=-\int_{0}^{t}\hspace{-2mm}\int_{{\Gamma_0}}\Psi\big(Y_s(y)\frac{\mathrm{dvol}_{g}}{\mathrm{dvol}_{g^s}}(y)\big)\varphi(s,y)\mathrm{dvol}_{\Gamma_0}(dy)ds\\+&\sum_{k=1}^{\infty}\int_{0}^{t}\big(\tilde{\sigma}_k(s,Y_s),\varphi(s)\big)_sd\beta^k_s+\int_{0}^{t}\big(Y_s,\dot{\varphi}(s)\big)_sds+\int_{0}^{t}\big(Y_s,\Phi(s)\varphi(s)\big)ds,
	\end{align}
	for any $\varphi$ obtained through (\ref{05092136}). On the other hand, for any $\varphi\in C_c^1\big([0,T),\dot{L}^p(\Gamma_0)\big)$, $\psi(\cdot)\ \dot{=}\ R_{-\cdot}\varphi(\cdot)\in C_c^1([0,T), \dot{H}^{2,p}(\Gamma_0))$. This follows from the property
	$$ R_{\cdot}\in C^1\Big([0,T), \mathcal{L}\big(\dot{H}^{2,p}(\Gamma_0),\dot{L}^p(\Gamma_0)\big) \Big)\text{ and }R_{-\cdot}\in C^1\Big([0,T), \mathcal{L}\big(\dot{L}^p(\Gamma_0),\dot{H}^{2,p}(\Gamma_0)\big) \Big),$$
	due to the regularity in time of the elliptic operator $R_\cdot$ , which is given locally in (\ref{05092203}), and the fact that the norm $\|\cdot\|_{2,p}$ constructed in (\ref{05092209}) is equivalent to the $H^{2,p}(\Gamma_0)$ norm defined by (\ref{05081740}) and (\ref{07020120}). Since $C^2(\Gamma_0)$ is dense in ${H}^{2,p}(\Gamma_0)$, we can construct a sequence $\psi_n\in C_c^2([0,T], C^2(\Gamma_0))\subseteq C^2([0,T]\times\Gamma_0)$ such that $\psi_n$ converges to $\psi$ in $H^{1,p}\big({[0,T], {H}^{2,p}(\Gamma_0)}\big)$. Then we have $$\varphi_n\ \dot{=} R_\cdot\big(\psi_n(\cdot)-[\psi_n(\cdot)]_0 \big)\rightarrow\varphi\text{ in } H^{1,p}([0,T],\dot{L}^p(\Gamma_0)).$$
	By $(i)$ of Definition \ref{ws}, (\hyperlink{C1}{C1}), (\ref{05082011}),  ($\hyperlink{C3}{C3}$), Remark \ref{05112043} and (\ref{05092313}), we take $\varphi=\varphi_n$ in (\ref{05092236}) and let $n\rightarrow\infty$ to deduce that (\ref{05092236}) holds almost surely for any $\varphi\in C_c^1\big([0,T),\dot{L}^p(\Gamma_0)\big)$. Therefore, (\ref{05092236}) is equivalent to the following stochastic equation on $V^*$ in the variational setting,
	\begin{numcases}{} \label{05100000}
		\nonumber d\iota_t^* Y_t=\widetilde{A}(t,Y_t)dt\ + \iota_t^*\tilde{B}(t,Y_t)dW_t+\Phi(t)Y_tdt,\text{ $t\in[0,T]$},\\ Y_0=x_0\in \dot{H}^{-1}(\Gamma_0),
	\end{numcases}
	where $\tilde{B}(\cdot,\cdot):[0,T]\times V\rightarrow L_2(l^2,H)$ is defined as $\tilde{B}(t,u)=\{\tilde{\sigma}_k(t,u)\}_{k\geq 1}$, for $u\in V=\dot{L}^2(\Gamma_0)$ and $W=\{\beta_k\}_{k\geq 1}$ represents an $l^2$-cylindrical Brownian motion,   $\widetilde{A}(\cdot,\cdot):[0,T]\times V\rightarrow V^*$ is given by
	\begin{align*}
		\langle \widetilde{A}(s,u),v\rangle=-\int_{\Gamma_0}\Psi\big((u\frac{\mathrm{dvol}_g}{\mathrm{dvol}_{g^s}})(y)\big)v(y)\mathrm{dvol}_{\Gamma_0}(dy),
	\end{align*}
	for $v\in V$. Notably, by ($\hyperlink{C3}{C3}$) and $(i)$ of Definition \ref{ws}, we deduce that
	$\widetilde{A}(\cdot,Y_\cdot)\in L^\frac{p}{p-1}
	\big(\Omega\times [0,T],  V^*\big)$. Let's denote $\iota_{-\cdot} \widetilde{A}(\cdot,u)$ by $A(\cdot,u)$ for $u\in V$. By the same argument for deducing (\ref{01280342}) from (\ref{05110042}), we conclude that $Y_\cdot$ is the solution of the following SPDE in the Gelfand triple $\dot{L}^p(\Gamma_0)\hookrightarrow \dot{H}^{-1}(\Gamma_0)\hookrightarrow\big(\dot{L}^p(\Gamma_0)\big)^*$:
	\begin{numcases}{} \label{05100037}
		\nonumber dY_t=A(t,Y_t)dt\ + \tilde{B}(t,Y_t)dW_t,\text{ $t\in[0,T]$},\\ Y_0=x_0\in \dot{H}^{-1}(\Gamma_0).
	\end{numcases}
	Thus, we have shown that any solution $X$ to equation (\ref{pme}) can be transformed into a solution $Y$ of the SPDE (\ref{05100037}) by  $Y_s\ \dot{=}\ \widetilde{X}^{s,*}_s$.
	\par
	\vskip 0.3cm
	Conversely, if $Y_t$ be a solution to equation (\ref{05100037}) in the sense of Definition \ref{01252114} and let $X_t\ \dot{=}\ \bar{Y}_t^{t,*}$. Take any $\phi\in C^2(Q_T)$ and denote $\phi\circ \widetilde{G}\in C^2([0,T]\times\Gamma_0)$ by $\widetilde{\phi}$,  $R_\cdot\big(\widetilde{\phi}(\cdot)-[\widetilde{\phi}(\cdot)]_0\big)$ by $\varphi(\cdot)$. Applying Theorem \ref{ito}, we obtain
	\begin{align}\label{05100256}
		\nonumber	&\ \big(Y_t,\varphi(t)\big)_t\\=&\ \nonumber \big(Y_0,\varphi(0)\big)_0+\int_{0}^{t}\langle \widetilde{A}(s,Y_s),\varphi(s)\rangle ds+\int_{0}^{t}\Big(\widetilde{B}(s,Y_s)dW_s,\varphi(s)\Big)_s+\int_{0}^{t}\big( Y_s,\dot{\varphi}(s)\big)_s ds\\&\ +\int_{0}^{t}\big(\Phi(s)Y_s,\varphi(s)\big)ds.
	\end{align}
	Note $$\dot{\varphi}(s)=R_s\dot{\big(\widetilde{\phi}(\cdot)-[\widetilde{\phi}(\cdot)]_0\big)}(s)+\frac{dR_s}{ds}\big(\widetilde{\phi}(s)-[\widetilde{\phi}(s)]_0\big)\in C\big([0,T],\dot{H}^{-1}(\Gamma_0)\big).$$
	Thus
	\begin{align}\label{05100209}
		\nonumber&\int_{0}^{t}\big( Y_s,\dot{\varphi}(s)\big)_s ds\\=&\nonumber\int_{0}^{t}(Y_s,R_s\dot{\big(\widetilde{\phi}(s)-[\widetilde{\phi}(s)]_0\big)}(s))_sds+\int_{0}^{t}\big(Y_s,\frac{dR_s}{ds}\big(\widetilde{\phi}(s)-[\widetilde{\phi}(s)]_0\big)\big)_sds\\=&\int_{0}^{t}\ _{{H}^{-1}(\Gamma_0)}\langle Y_s,\dot{\big(\widetilde{\phi}(\cdot)-[\widetilde{\phi}(\cdot)]_0\big)}(s)\rangle_{{H}^{1}(\Gamma_0)} ds-\int_{0}^{t}\ \big(\Phi(s)Y_s,\varphi(s)\big) ds\nonumber\\=&\int_{0}^{t}\ _{{H}^{-1}(\Gamma_0)}\langle Y_s,\dot{\big(\widetilde{\phi}(s)\big)}\rangle_{{H}^{1}(\Gamma_0)} ds-\int_{0}^{t}\ \big(\Phi(s)Y_s,\varphi(s)\big) ds.
	\end{align}
	By (\ref{05100256}) and (\ref{05100209}), as well as (\ref{05100141})-(\ref{05100143}), we obtain 	(\ref{05082045}). Since $X_t\ {=}\ \bar{Y}_t^{t,*}$ and according to (\ref{05100247}), Remark \ref{05112043} and (\ref{05082036}), we can see that the condition (iii) of Definition \ref{ws} holds for $\phi\in C^2(Q_T)$, representing the weak formulation of the stochastic Stefan problem on moving hypersurface. The other conditions (i) and (ii) in the definition 3.5 follow straightforwardly. Thus the stochastic process ${X}$ on moving hypersurfaces, obtained by $X_t\ \dot{=}\ \bar{Y}_t^{t,*}$, where $Y$ is the variational solution of (\ref{05100037}), is a solution of ($\ref{pme}$).
	\subsection{Verification of (\protect\hyperlink{H1}{H1})-(\protect\hyperlink{H5}{H5}) }
	Since we have established the correspondence  between the solutions to equations (\ref{pme}) and (\ref{05100037}). To obtain the well-posedness of the stochastic Stefan problem (\ref{pme}), it suffices to establish the well posedness of the SPDE (\ref{05100037}). In order to apply  Theorem \ref{wp}, it remains to verify conditions (\protect\hyperlink{H1}{H1})-(\protect\hyperlink{H5}{H5}) for $A(\cdot,\cdot)$ and $\tilde{B}(\cdot,\cdot)$ constructed in  (\ref{05100037}). We first note  that (\ref{lp}) and (\ref{lg})  imply that for any $u,v\in V$ and a.e. $t\in[0,T]$,
	\begin{align*}
		&\|\widetilde{B}(t,v)\|_{L_2(l^2,H)}^2\leq Cf(t)(\|v\|_H^2+1)\text{ and }\\& \|\widetilde{B}(t,v)-\widetilde{B}(t,u)\|_{L_2(l^2,H)}^2\leq Cf(t)\|v-u\|_H^2.
	\end{align*}
	Next we verify  (\protect\hyperlink{H1}{H1})-(\protect\hyperlink{H4}{H4}) for $\iota^*_{t}A(t,\cdot)=\widetilde{A}(t,\cdot)$.
	\\ \noindent {\bf Verification of (\protect\hyperlink{H1}{H1})}: for any $u,v,w\in V$,  $s\in[0,T]$ and $\lambda\in \mathbb{R}$, we have
	\begin{align*}
		\lambda\rightarrow\langle \widetilde{A}(s,u+\lambda v),\omega\rangle=-\int_{{\Gamma_0}}\Psi\Big((u+\lambda v)\frac{\mathrm{dvol}_g}{\mathrm{dvol}_{g^s}}(y)\Big) w(y)\mathrm{dvol}_{\Gamma_0}(dy)
	\end{align*}
	is continuous, by the dominated convergence theorem,  (\hyperlink{$\Psi$1}{$\Psi$1}) and (\hyperlink{$\Psi$4}{$\Psi$4}).\par
	\noindent {\bf Verification of (\protect\hyperlink{H2}{H2})}: for any $u,v\in V$ and  $s\in[0,T]$, by (\hyperlink{$\Psi$2}{$\Psi$2}),
	\begin{align*}
		&\langle \widetilde{A}(s,u)-\widetilde{A}(s,v),u-v\rangle\\=&-\int_{{\Gamma_0}}\Big(\Psi\big( (u \frac{\mathrm{dvol}_g}{\mathrm{dvol}_{g^s}})\big)(y)-\Psi\big( (v \frac{\mathrm{dvol}_g}{\mathrm{dvol}_{g^s}})\big)(y)\Big)\cdot (u-v)(y) \mathrm{dvol}_{\Gamma_0}(dy)\\=&-\int_{{\Gamma_0}}\Big(\Psi\big( (u \frac{\mathrm{dvol}_g}{\mathrm{dvol}_{g^s}})\big)(y)-\Psi\big( (v \frac{\mathrm{dvol}_g}{\mathrm{dvol}_{g^s}})\big)(y)\Big)\cdot (u\frac{\mathrm{dvol}_g}{\mathrm{dvol}_{g^s}}-v\frac{\mathrm{dvol}_g}{\mathrm{dvol}_{g^s}})(y) \mathrm{dvol}_{g^s}(dy)\\\leq&\ 0.
	\end{align*}
	\noindent {\bf Verification of (\protect\hyperlink{H3}{H3})}:  For any $u\in V$ and $s\in[0,T]$, by (\ref{05101834}) and  (\hyperlink{$\Psi$3}{$\Psi$3}), we have for some constant $c>0$ and $C>0$,
	\begin{align*}
		\langle \widetilde{A}(s,u),u\rangle&=-\int_{{\Gamma_0}}\Psi\big( (u \frac{\mathrm{dvol}_g}{\mathrm{dvol}_{g^s}})\big)(y)\cdot (u\frac{\mathrm{dvol}_g}{\mathrm{dvol}_{g^s}})(y) \mathrm{dvol}_{g^s}(dy)\\&\leq -c\|u\|_V^p+C.
	\end{align*}
	\noindent {\bf Verification of (\protect\hyperlink{H4}{H4})}: for any $s\in[0,T]$ and $u\in V$, by (\ref{05101834}) and (\hyperlink{$\Psi$4}{$\Psi$4}), we have for some constant $C>0$,
	\begin{align*}
		\|\widetilde{A}(s,u)\|_{V^*}\leq \|\Psi(u\frac{\mathrm{dvol}_g}{\mathrm{dvol}_{g^s}})\|_{L^\frac{p}{p-1}({\Gamma_0})}\leq C\|u\|^{p-1}_{L^p(\Gamma_0)}+C.
	\end{align*}
	The well-posedness of the Stefan type problem (\ref{pme}) is thus proved.
	\subsection{Stochastic Transport formula for $\dot{H}^{-1}(\Gamma_\cdot)$ norm}
	As a corollary of the preceding subsections, we can derive the a stochastic transport formula for the $\|\cdot\|_{\dot{H}^{-1}(\Gamma_\cdot)}$ norm. We'd like to mention that we can apply similar method to deduce the stochastic transport formula for solutions to the equations mentioned at the begining of this section. Apply Theorem \ref{ito} to equation (\ref{05100037}), we obtain for any $t\in[0,T]$,
	\begin{align}\label{05122329}
		\nonumber&\ |Y_t|_t^2\\=&\nonumber\ |Y_0|_0^2-2\int_{0}^{t}\int_{{\Gamma_0}}\Psi\big(Y_s\frac{\mathrm{dvol}_g}{\mathrm{dvol}_{g^s}}\big)Y_s(y)\mathrm{dvol}_{\Gamma_0}(dy)ds+2\sum_{k=1}^{\infty}\int_{0}^{t}\big(\widetilde{\sigma}_k(s,Y_s),Y_s\big)_sd\beta^k_s\\&\ +\sum_{k=1}^{\infty}\int_{0}^{t}|\widetilde{\sigma}_k(s,Y_s)|_s^2ds+\int_{0}^{t}(\Phi(s)Y_s,Y_s)ds.
	\end{align}
	By (\ref{05100141})-(\ref{05100143}) and (\ref{05082011}), (\ref{05122329}) is written as
	\begin{align}\label{05122330}
		\nonumber&\ |X_t|_{\dot{H}^{-1}(\Gamma_t)}^2\\=&\nonumber\ |X_0|_{\dot{H}^{-1}(\Gamma_0)}^2-2\int_{0}^{t}\int_{{\Gamma_s}}\Psi\big(X_s(x)\big)X_s(x)\mathrm{dvol}_{\Gamma_s}(dx)ds+2\sum_{k=1}^{\infty}\int_{0}^{t}\big(\sigma_k(s,X_s),X_s\big)_{\dot{H}^{-1}(\Gamma_s)}d\beta^k_s\\&\ +\sum_{k=1}^{\infty}\int_{0}^{t}|{\sigma}_k(s,Y_s)|_{\dot{H}^{-1}(\Gamma_s)}^2ds+\int_{0}^{t}(\Phi(s)Y_s,Y_s)ds.
	\end{align}
	Recall by (\ref{05092048}),
	\begin{align*}
		(\Phi(s)Y_s,Y_s)
		=-\sum_{\alpha=1}^{N}\int_{{U^\alpha}}\chi^\alpha(X^\alpha(\theta))\frac{\partial }{\partial \theta^i}[R_{-s}Y_s]\frac{\partial }{\partial \theta^j} [R_{-s}Y_s]\frac{d}{dt}\big(\sqrt{g^\alpha(s,\theta)}g_{s,\alpha}^{ij}(\theta)\big)d\theta.
	\end{align*}
	We can also show $\overline{R_{-s}Y_s}^{s}=(-\Delta_{\Gamma_s})^{-1}X_s\in \dot{H}^1(\Gamma_s)$ for any $s\in[0,T]$. Therefore  (\ref{05122330}) can be expressed as
	\begin{align*}
		\nonumber&\ |X_t|_{\dot{H}^{-1}(\Gamma_t)}^2\\=&\nonumber\ |X_0|_{\dot{H}^{-1}(\Gamma_0)}^2-2\int_{0}^{t}\int_{{\Gamma_s}}\Psi\big(X_s(x)\big)X_s(x)\mathrm{dvol}_{\Gamma_s}(dx)ds+2\sum_{k=1}^{\infty}\int_{0}^{t}\big(\sigma_k(s,X_s),X_s\big)_{\dot{H}^{-1}(\Gamma_s)}d\beta^k_s\\&\ +\sum_{k=1}^{\infty}\int_{0}^{t}|{\sigma}_k(s,Y_s)|_{\dot{H}^{-1}(\Gamma_s)}^2ds-\int_{0}^{t}\int_{{\Gamma_s}}\mathcal{B}(s,v)\nabla_{\Gamma_s}(-\Delta_{\Gamma_s})^{-1}X_s\cdot \nabla_{\Gamma_s}(-\Delta_{\Gamma_s})^{-1}X_s\mathrm{dvol}_{\Gamma_s}(dx)ds,
	\end{align*}
	where $\mathcal{B}(s,v)$ is the tensor given in (5.7) of \cite{DE} with $\mathcal{A}=I$. To be precise,
	\begin{align*}
		B(s,v)=\nabla_{\Gamma_s}\cdot v(s,\cdot)\ I_{(n+1) \times (n+1)}-2D(s,v),
	\end{align*}
	where for $1\leq i,j\leq n+1$,
	\begin{align}
		D(s,v)_{ij}=\frac{1}{2}\Big(\big(\nabla_{\Gamma_s}\big)_i v_j(s,\cdot)+\big(\nabla_{\Gamma_s}\big)_j v_i(s,\cdot)\Big).
	\end{align}
	\appendix
	\renewcommand{\theequation}{A.\arabic{equation}}
	\section{A variational formulation of $\iota_{-t}^*$}
	\setcounter{equation}{0}
	\numberwithin{theorem}{section}
	\begin{lemma}\label{app}
		Let $\{\iota_t^*\}_{t\in[0,T]}$ be a family of the bounded operators on $H$ obtained from a family of inner products $\{(\cdot,\cdot)_t\}_{t\in[0,T]}$ satisfying (\hyperlink{C1}{C1})-(\hyperlink{C4}{C4}), then it holds for any $x,y\in H$, $t\in[0,T]$,
		\begin{align}
			\big(\iota_{-t}^*x,y\big)=(x,y)-\int_{0}^{t}\big(\iota_{-s}^*\Phi(s)\iota_{-s}^*x,y\big)ds.
		\end{align}
	\end{lemma}
	\begin{proof}
		For any $x,y\in H$,
		\begin{align*}
			(\iota_{-t}^*x,y)-(\iota_{-s}^*x,y)=(\iota_{-t}^*x-\iota_{-s}^*x,y)=\big(\iota_{-s}^*(\iota_s^*-\iota_t^*)\iota_{-t}^*x,y\big).
		\end{align*}
		Thus for any finite collection of disjoint intervals $\{[s_i,t_i]\}_{i=1}^N$, we have
		\begin{align*}
			\sum_{i=1}^{N}|(\iota_{-t_i}^*x,y)-(\iota_{-s_i}^*x,y)|=\sum_{i=1}^{N}|\big(\iota_{-s_i}^*(\iota_{s_i}^*-\iota_{t_i}^*)\iota_{-t_i}^*x,y\big)|.
		\end{align*}
		By (ii) and (iii) of Lemma \ref{01201913}, we have
		\begin{align*}
			\sum_{i=1}^{N}|(\iota_{-t_i}^*x,y)-(\iota_{-s_i}^*x,y)|\leq c_1^4|x||y|\sum_{i=1}^{N}\int_{s_i}^{t_i}\|\Phi(r)\|dr.
		\end{align*}
		By (\hyperlink{C2}{C2}), for any $\varepsilon$, there exists $\delta>0$ such that for any finite collection of disjoint intervals $\{[s_i,t_i]\}_{i=1}^N$ with $\sum_{i=1}^{N}|t_i-s_i|\leq\delta$, then we have
		\begin{align*}
			\sum_{i=1}^{N}|(\iota_{-t_i}^*x,y)-(\iota_{-s_i}^*x,y)|\leq \varepsilon.
		\end{align*}
		This implies that for any $x,y\in H$, the function $t\rightarrow(\iota_{-t}^*x,y)$ is absolutely continuous, implying that $t\rightarrow(\iota_{-t}^*x,y)$ is differentiable a.e. Thus it suffices to show that its derivative almost everywhere equals to $\big(\iota_{-\cdot}^*\Phi(\cdot)\iota_{-\cdot}^*x,y) $.\par
		To this end, we consider a countable dense subset of $H$ denoted by $\{x_i\}_{i\geq 1}$, and subsets $\Gamma,\Gamma_1,\Gamma_2\subseteq[0,T]$ defined by
		\begin{align*}
			&\Gamma_1=\scalebox{1.5}[1.5]{$\cap$}_{i,j=1}^\infty \big\{t\in[0,T]:\text{ $t$ is the Lebesgue point of $s\in[0,T]\rightarrow\big(\iota_s^*x_i,x_j\big)$.}\big\}\\&\Gamma_2=\big\{t\in[0,T]:\text{ $t$ is the Lebesgue point of $s\in[0,T]\rightarrow\|\Phi(s)\|$.}\big\}\\&\Gamma=\Gamma_1\cap\Gamma_2.
		\end{align*}
		Clearly, $|\Gamma|=T$.
		We claim that: For any $x,y\in H$, $s\in\Gamma$,
		\begin{align*}
			t\in[0,T]\rightarrow(\iota_t^*x,y)\text{ is differentiable at $s$, with derivative $\big(\Phi(s)x,y\big)$.}
		\end{align*}
		We first prove that for any $j\in\mathbb{N}$, $x\in H$
		\begin{align}
			\lim\limits_{t\rightarrow s}\frac{(\iota_t^*x,x_j)-(\iota_s^*x,x_j)}{t-s}=(\Phi(s)x,x_j).\label{01262034}
		\end{align}
		For any $\varepsilon>0$, there exists $i_0\in\mathbb{N}$ such that $|x-x_{i_0}|\leq\varepsilon$, then
		\begin{align*}
			&\ \frac{(\iota_t^*x,x_j)-(\iota_s^*x,x_j)}{t-s}-(\Phi(s)x,x_j)\\=&\ \Big\{\frac{(\iota_t^*x_{i_0},x_j)-(\iota_s^*x_{i_0},x_j)}{t-s}-(\Phi(s)x_{i_0},x_j)\Big\}+\Big\{(\Phi(s)x_{i_0},x_j)-(\Phi(s)x,x_j)\Big\} \\&\ +\frac{\big(\iota_t^*(x-x_{i_0}),x_j\big)-\big(\iota_s^*(x-x_{i_0}),x_j\big)}{t-s} \\=&\ \mathrm{I+II+III}.
		\end{align*}
		Since $s\in\Gamma$, we can deduce that as $t\rightarrow s$,
		\begin{align}\label{01261952}
			&\mathrm{I}\rightarrow 0,\\
			&|\mathrm{II}|\leq \|\Phi(s)\||x_{i_0}-x||x_j|\leq \varepsilon \|\Phi(s)\||x_j|,\\&
			|\mathrm{III}|\leq\frac{\varepsilon |x_j|}{|t-s|}\int_{s\wedge t}^{s\vee t}\|\Phi(r)\|dr\rightarrow \varepsilon|x_j|\|\Phi(s)\|.\label{01262018}
		\end{align}
		Thus, by first letting $t\rightarrow s$ and then letting $\varepsilon\rightarrow0$, we can obtain
		\begin{align*}
			\big| \frac{(\iota_t^*x,x_j)-(\iota_s^*x,x_j)}{t-s}-(\Phi(s)x,x_j)\big|\rightarrow 0,
		\end{align*}
		which implies (\ref{01262034}). Using the same method, we can also show that for any $s\in\Gamma$, $x,y\in H$,
		\begin{align}
			\lim\limits_{t\rightarrow s}\frac{(\iota_t^*x,y)-(\iota_s^*x,y)}{t-s}=(\Phi(s)x,y).\label{01262036}
		\end{align}
		Now, for $s\in\Gamma$,
		\begin{align*}
			&\big(\frac{\iota_{-t}^*x-\iota_{-s}^*x}{t-s}+\iota^*_{-s}\Phi(s)\iota^*_{-s}x,y\big)=\ -\Big(\iota_{-s}^*\big(\frac{(\iota_t^*-\iota_s^*) \iota_{-t}^*}{t-s}x-\Phi(s)\iota_{-s}^*\big)x,y\Big)\\=\ &-\Big(\iota_{-s}^*\frac{(\iota_t^*-\iota_s^*) (\iota_{-t}^*-\iota_{-s}^*)}{t-s}x,y\Big)-\Big(\big(\iota_{-s}^*\frac{(\iota_t^*-\iota_s^*)\iota_{-s}^*}{t-s}x,y\big)-\big(\iota_{-s}^*\Phi(s)\iota_{-s}^*x,y\big)\Big)\\=\ &-\Big(\frac{(\iota_t^*-\iota_s^*) (\iota_{-t}^*-\iota_{-s}^*)}{t-s}x,\iota_{-s}^*y\Big)-\Big(\big(\frac{(\iota_t^*-\iota_s^*)\iota_{-s}^*}{t-s}x,\iota_{-s}^*y\big)-\big(\Phi(s)\iota_{-s}^*x,\iota_{-s}^*y\big)\Big)\\=\ &\mathrm{I+II}.
		\end{align*}
		Since $s\in\Gamma$, by (\ref{01262036}), we can see that $|\mathrm{II}|\rightarrow 0$ as $t\rightarrow s$. Moreover,
		\begin{align*}
			|\mathrm{I}|&\leq \big|\Big(\frac{(\iota_t^*-\iota_s^*) (\iota_{-t}^*-\iota_{-s}^*)}{t-s}x,\iota_{-s}^*y\Big)\big|\leq\frac{c_1^2|x||y|}{|t-s|}\|\iota_{-t}^*-\iota_{-s}^*\|\int_{s}^{t}\|\Phi(r)\|dr\\&=\frac{c_1^2|x||y|}{|t-s|}\|\iota_{-t}^*\big(\iota_t^*-\iota_s^*\big)\iota_{-s}^* \|\big(\int_{s}^{t}\|\Phi(r)\|dr\big)\\&\leq \frac{c_1^6|x||y|}{|t-s|}\big(\int_{s}^{t}\|\Phi(r)\|dr\big)^2.
		\end{align*}
		Since $s\in\Gamma$, $|\mathrm{I}|\rightarrow 0$ as $t\rightarrow s$. Thus, the proof is complete.
	\end{proof}
\vskip 0.3cm
	\noindent {\large{\bf{{Acknowledgement}}}}
	This work is partially supported by National Key R\&D Program of China (No. 2022 YFA1006001), National Natural Science Foundation of China (No. 12131019, 12371151, 11721101), and the Fundamental Research Funds for the Central Universities(No. WK3470000031, WK0010000089).
	\vskip 0.3cm
	\noindent {\large{\bf{{Data availability}}}}
	Data sharing is not applicable to this article as no datasets were generated or analysed during the current study.
	\vskip 0.4cm
	\noindent {\Large{\bf{{Declarations}}}}
	\vskip 0.4cm
	\noindent {\large{\bf{{Conflict of Interest}}}} On behalf of all authors, the corresponding author states that there is no conflict of interest.
	
\end{document}